\documentclass[a4paper,11pt]{article}
\usepackage[latin1]{inputenc}
\usepackage[T1]{fontenc}
\usepackage[british,UKenglish,USenglish,american]{babel}
\usepackage{amsmath}
\usepackage{amsfonts}
\usepackage{hyperref}
\usepackage[T1]{fontenc}
\usepackage[latin1]{inputenc}
\usepackage{theorem}  
\usepackage{graphics}
\usepackage{makeidx}
\usepackage{fancyhdr}
\usepackage{bbm} 
\usepackage{bm}
\usepackage{amssymb}
\usepackage{color}
\numberwithin{equation}{section}

\newcommand{\be}{\begin{equation}}
\newcommand{\ee}{\end{equation}}
\newcommand{\benn}{\begin{equation*}}
\newcommand{\eenn}{\end{equation*}}
\newcommand{\bea}{\begin{eqnarray}}
\newcommand{\eea}{\end{eqnarray}}
\newcommand{\beann}{\begin{eqnarray*}}
\newcommand{\eeann}{\end{eqnarray*}}

\newtheorem{theorem}{Theorem}[section]

\newtheorem{corollary}[theorem]{Corollary}
\newtheorem{lemma}[theorem]{Lemma}

\newtheorem{definition}[theorem]{Definition}
\theorembodyfont{\rm}

\newtheorem{remark}[theorem]{Remark}
\newtheorem{example}[theorem]{Example}

\newcommand{\qed}{\hfill $\Box$\smallskip}

\def\R{\mathbb{R}}
\def\C{\mathbb{C}}
\def\N{\mathbb{N}}

\def\cD{\mathcal{D}}

\def\txtd{{\textnormal{d}}}

\def\txtD{{\textnormal{D}}}

\title{Random  attractors  via pathwise mild solutions\\
for stochastic parabolic evolution equations}

\author{Christian Kuehn\thanks{Technical University of Munich (TUM), 
		Faculty of Mathematics, 85748 Garching bei M\"unchen, Germany. E-Mail: ckuehn@ma.tum.de},~~~Alexandra 
	Neam\c tu\thanks{Bielefeld University, Faculty of Mathematics, 33615 Bielefeld, Germany. E-Mail: alexandra.neamtu@uni-bielefeld.de},~~~Stefanie Sonner\thanks{Radboud University Nijmegen, IMAPP--Mathematics, 6500 GL Nijmegen, The Netherlands. E-Mail: s.sonner@ru.nl}
	}

\begin{document}

\maketitle
\begin{abstract}
We investigate the longtime behavior of stochastic partial differential 
equations (SPDEs) with differential operators that depend on time and the underlying probability space. 
In particular, we consider stochastic parabolic evolution problems in Banach spaces with additive noise and prove
 the existence of  random exponential 
attractors. These are compact  random sets of finite fractal dimension that 
contain the global random attractor and are attracting at an 
exponential rate.
In order to apply the framework of random dynamical systems, we use the concept of pathwise mild solutions. 
\end{abstract}

\section{Introduction}\label{intro}

Our aim is to study the longtime dynamics of stochastic evolution equations  using an approach that is different from the 
classical one. Namely, instead of transforming the SPDE into a random PDE,  
we work with solutions that are defined pathwise, see~\cite{BessaihGarridoSchmalfuss, PronkVeraar}. We consider parabolic problems with random differential operators and use a pathwise representation formula to show that the 
solution operator generates a random dynamical system and to prove that it possesses random attractors of finite fractal dimension.
\medskip

In particular, let $X$ be a separable Banach space and let $(\Omega, \mathcal{F}, \mathbb{P})$ be a complete probability space with filtration $(\mathcal{F}_t)_{t\in\R}$. We consider stochastic parabolic evolution equations of the form 
\begin{align}\label{problem}
\txtd u(t)&=(A(t,\omega)u(t)+F(u(t)))\txtd t+\sigma  \txtd W(t),  
\end{align}
where $(A(t,\omega))_{t\in\R,\omega\in\Omega}$ is a measurable, adapted family of sectorial 
operators in $X$ depending on time and the underlying probability space. Moreover, $F$ is the nonlinearity, $\sigma>0$ indicates the noise intensity,  and $(W(t))_{t\geq 0}$ denotes an $X$-valued Brownian motion. 
\medskip

The common approach to show the existence of random attractors is to introduce a suitable change of variables that 
transforms the SPDE into a family of PDEs with random coefficients.
The resulting random PDEs can be studied by deterministic techniques. 
This method has been applied to a large variety of PDEs, mainly for equations perturbed by additive noise or by a particular linear multiplicative noise, e.g. see
\cite{schmalfuss,Gess1,GessLiuRoeckner,Bates,Wang} and the references therein. However, for more general situations such a change of variables is 
not always known or cannot be performed.
In~\cite{Gess2}, using the theory of strongly monotone operators,  a strictly stationary solution of the equation
 \begin{align*}
 \txtd u(t) = A(t,\omega, u(t))~\txtd t +\sigma\txtd W(t)
 \end{align*}
 was constructed. This allows to transform SPDEs of the form \eqref{problem} into a family of random PDEs. Using this ansatz, the existence of random attractors was shown in \cite{Gess2} for a class of SPDEs including equations such as \eqref{problem}.\\ 
 Here, we follow a different approach. We aim to use the notion of pathwise mild solutions introduced by Pronk and Veraar in \cite{PronkVeraar} 
 to establish the existence of global and exponential random attractors for~\eqref{problem}.
 So far, only few results concerning the existence of exponential attractors for SPDEs have been obtained.
 In particular, let  $(U(t,s,\omega))_{t\geq s,\omega\in\Omega}$ be the stochastic evolution system generated by the family $(A(t,\omega))_{t\in\R,\omega\in\Omega}$, then the pathwise mild solution of \eqref{problem} is defined as 
\begin{align*}
u(t)=&\ U (t,0) u_{0} + \sigma U(t,0) W(t) + \int\limits_{0}^{t} U(t,s)F(u(s))~\txtd s\\
&\ -\sigma \int\limits_{0}^{t}U(t,s)A(s) (W(t) -W(s)) ~\txtd s,
\end{align*}
where, for simplicity, we omit to write the dependency of $A$ and $U$ on $\omega$.
This formula is motivated by formally applying integration by 
parts for the stochastic integral and it, indeed, yields a pathwise representation for the solution, see
\cite{PronkVeraar}. Instead, if one directly used the classical mild formulation of SPDEs to 
define a solution, the resulting stochastic integral would not be well-defined (see Subsection \ref{subs_spdes}).
\medskip

Our aim is to show that problem
\eqref{problem} generates a random dynamical system using the concept of pathwise mild solutions and to
prove the existence of random attractors. 
We will not only consider (global) random attractors, but also show that random exponential attractors exist. In particular, the existence of  random exponential attractors immediately implies the existence and finite fractal dimension of the (global) random attractor. To this end, we employ a general existence result for random exponential attractors 
in \cite{CaSo} which turns out to be easily applicable in our setting.
\medskip

Stochasticity plays an important role in many real world applications. Complex systems in physics, engineering or biology 
can be described by PDEs with coefficients that depend on stochastic processes. These random terms quantify the lack of knowledge of certain parameters in the equation or reflect external fluctuations. 
Problem \eqref{problem} is a semilinear parabolic problem where the coefficients of the differential operators $(A(t,\omega))_{t\in\mathbb{R},\omega\in\Omega}$ depend on a stochastic process with suitable properties, and the equation is perturbed by additive noise.  
A related, but simpler setting are random parabolic equations of the form
\begin{align}\label{eq1}
\txtd u(t) =( A(t,\omega) u(t)+F(t,\omega,u(t)))\txtd t.
\end{align}
The longtime behavior of such random evolution equations 
has been investigated using the random dynamical system approach in~\cite{CDLS,lns, MShen, SalakoShen}. To this end, the following structure of the random generators was assumed, 
$$
A(t,\omega):=A(\theta_t\omega)\qquad \forall t\in\R, \omega\in\Omega, 
$$
where $(\Omega,\mathcal{F}, \mathbb{P}, (\theta_t)_{t\in\mathbb{R}} )$ is an ergodic metric dynamical system. In this context, results concerning invariant manifolds~\cite{CDLS,lns}, principal Lyapunov exponents~\cite{MShen} and the stability of equilibria~\cite{SalakoShen} have been obtained. Random evolution equations of the form~\eqref{eq1} arise in several applications. For instance, setting $A(\theta_t\omega)u:=\Delta u + a(\theta_t\omega) u$ and $F(t,\omega,u):=-a(\theta_t\omega)u^2$, with a suitable measurable function $a:\Omega\to(0,\infty)$, 
we recover a random version of the Fisher-KPP equation,
$$
\txtd u(t) = [(\Delta  + a(\theta_t\omega)) u(t) - a(\theta_t\omega)u^2(t)  ]\txtd t,
$$ 
which was analyzed in~\cite{SalakoShen}.\\
In the present work, we consider equations of the form~\eqref{problem}, i.e.~we perturb a semilinear nonautonomous random parabolic equation by an infinite-dimensional Brownian motion and investigate the existence of random  attractors. 
\medskip

The outline of our paper is as follows. In Section~\ref{sec:p}, we collect basic notions and results from the theory of random dynamical systems and nonautonomous stochastic evolution equations and recall an existence result for random exponential attractors. In Section~\ref{ass:attr}, we formulate and prove our main results. First, we show that under suitable assumptions on $A$, $F$ and $W$, the solution operator corresponding to~\eqref{problem} generates a random dynamical system. Then, we establish the existence of an absorbing set and verify the so-called smoothing property of the random dynamical system.
These properties allow us to conclude the existence of random exponential attractors in Theorem~\ref{thm_expattr} and to derive upper bounds for their fractal dimension. As a consequence, the (global) random attractor 
exists and its fractal dimension is finite. 
In  Section~\ref{examples}, we provide explicit examples of nonautonomous random differential operators satisfying our hypotheses and point out potential applications of our results.
\medskip

Our paper provides a first, simple example that illustrates how the concept of pathwise mild solutions can be used to show the existence of global and exponential random attractors for SPDEs with random differential operators. 
Numerous extensions are imaginable. In particular, in future works we plan to relax the assumptions on the non-linear term $F$ and to consider Problem \eqref{problem} with multiplicative noise. Another interesting aspect would be to investigate the regularity of random attractors.

\section{Preliminaries}\label{sec:p}

We first collect some basic notions and results from the theory of random 
dynamical systems, which are mainly taken from~\cite{arnold,schmalfuss,crauel}.
Then, in Subsection \ref{subsec_expattr}, we state a general existence theorem for random exponential attractors which was proven in \cite{CaSo}. In Subsection \ref{subs_spdes}, we 
recall the notion of pathwise mild solutions 
for stochastic evolution equations introduced in \cite{PronkVeraar}.

\subsection{Random dynamical systems and random attractors}\label{subsec_rds}

In order to quantify uncertainty we describe an appropriate model of the noise.
If not further specified, $(\Omega,\mathcal{F},\mathbb{P})$ denotes a probability space. Moreover, $X$ is a separable and reflexive Banach space and $\|\cdot\|_X$ denotes the norm in $X$.

\begin{definition} 
Let $\theta:\mathbb{R}\times\Omega\rightarrow\Omega$  be a family of $\mathbb{P}$-preserving transformations (meaning that $\theta_{t}\mathbb{P}=\mathbb{P}$ for all $t\in\mathbb{R}$) with the following properties:
	\begin{description}
		\item[(i)] the mapping $(t,\omega)\mapsto\theta_{t}\omega$ is $(\mathcal{B}(\mathbb{R})\otimes\mathcal{F},\mathcal{F})$-measurable for all $t\in\R, \omega\in\Omega$;
		\item[(ii)] $\theta_{0}=\text{Id}_{\Omega}$;
		\item[(iii)] $\theta_{t+s}=\theta_{t}\circ\theta_{s}$ for all $t,s,\in\mathbb{R}$,
	\end{description}
	where $\mathcal{B}(\mathbb{R})$ denotes the Borel $\sigma$-algebra.
	Then, the quadruple $(\Omega,\mathcal{F},\mathbb{P},(\theta_{t})_{t\in\mathbb{R}})$ is called a metric dynamical system.
\end{definition}

\begin{remark}
	\begin{itemize}		
		\item[(a)] Here and in the sequel, we write $\theta_{t}\omega$ for $\theta(t,\omega),\ t\in\R,\omega\in\Omega$.
		\item[(b)] We always assume that $\mathbb{P}$ is ergodic with respect to
		$(\theta_{t})_{t\in\mathbb{R}}$, i.e. any $(\theta_{t})_{t\in\mathbb{R}}$-invariant 
		subset has either zero or full measure.
	\end{itemize}
\end{remark}

Our aim is to introduce a metric dynamical system associated to a two-sided $X$-valued Wiener process.\\
 
 
 For the sake of completeness, we recall the construction of such a process if $X$ is not a Hilbert space. First, we introduce an auxiliary separable Hilbert space $H$ and denote by $(W_H(t))_{t\geq 0}$ an $H$-cylindrical Brownian motion, i.e.~$(W_{H}(t)h)_{t\geq 0}$ is a real-valued Brownian motion for every $h\in H$ and $\mathbb{E}[W_{H}(t)h\cdot W_H(s)g]=\min\{s,t\}[h,g]_{H}$ for $s,t\geq 0$ and $h,g\in H$, where $[\cdot,\cdot]$ denotes the inner product in $H$.
Furthermore, an operator $G:H\to X$ is called $\gamma$-radonifying if 
\begin{equation*}
\mathbb{E}\left\|\sum\limits_{n=1}^{\infty}\gamma_{n}G e_{n} \right\|_X^{2}<\infty,
\end{equation*}
where $(\gamma_{n})_{n\in\N}$ is a sequence of independent standard Gaussian 
random variables  and $(e_{n})_{n\in\N}$ is an 
orthonormal basis in $H$.
If $X$ is isomorphic to $H$, then the previous condition means that $G$ is a Hilbert-Schmidt operator (notation: $G\in\mathcal{L}_{2}(H)$).
In this framework, letting $(\tilde e_n)_{n\in\N}$ be an orthonormal basis 
of  $(\mbox{ker}G)^{\perp}$, we know according to~\cite[Prop.~8.8]{isem} that the series
$$ \sum\limits_{n=1}^{\infty}W_{H}(t)\tilde e_nGe_n$$ converges
almost surely and defines an $X$-valued Brownian motion. Its covariance operator is given by $tGG^{*}$, where $G^*$ denotes the adjoint. Moreover, every $X$-valued Brownian motion can be obtained in this way.
Again, if $X$ is isomorphic to $H$ and $G\in\mathcal{L}_{2}(H)$ with $\|G\|_{\mathcal{L}_{2}(H)}=\mbox{Tr}(G G^{*})$, then the previous definition entails a trace-class Wiener process.  Finally, we extend this to a two-sided process in the standard way.
\medskip

To obtain a metric dynamical system associated to such a process, we let $C_{0}(\mathbb{R};X)$ denote the set of continuous $X$-valued functions which are zero at zero  
equipped with the compact open topology. We take $\mathbb{P}$ as the Wiener measure on 
$\mathcal{B}(C_{0}(\mathbb{R};X))$ having a covariance operator $Q$ on $X$. Then, 
Kolmogorov's theorem about the existence of a continuous version yields the canonical probability space 
$(C_{0}(\mathbb{R};X),\mathcal{B}(C_{0}(\mathbb{R};X)),\mathbb{P})$. Moreover, to obtain an ergodic metric 
dynamical system we introduce the Wiener shift, which is defined as
	\begin{align}\label{shift}
	\theta_{t}\omega(\cdot{}):=\omega(t+\cdot{})-\omega(t)\quad \mbox{for all } t\in\R, \omega\in C_{0}(\mathbb{R};X).
	\end{align}
Throughout this manuscript, $\theta_t\omega(\cdot)$ will always denote the Wiener shift.\\		

We now recall the definition of a random dynamical system.
\begin{definition} 
A continuous random dynamical system on $X$ over a metric dynamical system $(\Omega,\mathcal{F},\mathbb{P},(\theta_{t})_{t\in\mathbb{R}})$ is a mapping $$\varphi:\mathbb{R^{+}}\times\Omega\times X\to X,\mbox{  } (t,\omega,x)\mapsto \varphi(t,\omega,x), $$
	which is $(\mathcal{B}(\mathbb{R}^{+})\otimes\mathcal{F}\otimes\mathcal{B}(X),\mathcal{B}(X))$-measurable and satisfies:
	\begin{description}
		\item[(i)] $\varphi(0,\omega,\cdot{})=\emph{Id}_{X}$ for all $\omega\in\Omega$;
		\item[(ii)]$ \varphi(t+\tau,\omega,x)=\varphi(t,\theta_{\tau}\omega,\varphi(\tau,\omega,x)), \mbox{ for all } x\in X, \mbox{ } t,\tau\in\mathbb{R}^{+}\mbox{ and all }\omega\in\Omega;$
		\item[(iii)] $\varphi(t,\omega,\cdot{}):X\to X$ is continuous for all $t\in\mathbb{R}^{+}$ and $\omega\in\Omega$.
	\end{description}
\end{definition}

The second property is referred to as the \emph{cocycle property} and generalizes the semigroup property. 
In fact, if $\varphi$ is independent of $\omega$, (\textbf{ii}) reduces exactly to the semigroup property, i.e.  $\varphi(t+\tau,x)=\varphi(t,\varphi(\tau,x))$. 
For random dynamical systems, the evolution of the noise $(\theta_{t}\omega)$ has additionally to be taken 
into account. 

Under suitable assumptions, the solution operator of a random differential equation generates a 
random dynamical system.
Stochastic (partial) differential equations are more involved since stochastic integrals are defined 
\emph{almost surely}, though the cocycle property must hold \emph{for all} $\omega$.
\medskip

Referring to the monograph by Arnold~\cite{arnold}, 
it is well-known that stochastic (ordinary)  differential equations generate random dynamical systems under suitable assumptions on the coefficients. This is due to the flow property, see~\cite{Kunita} which can be deduced from Kolmogorov's theorem about the existence of a (H\"older-) continuous random field with a finite-dimensional parameter range. Here, the parameters of this random field are the time and the non-random initial data.\\
Whether an SPDE generates a random dynamical system has been a long-standing open problem, since Kolmogorov's theorem breaks down for random fields parametrized by infinite-dimensional Hilbert spaces,~see~\cite{Mohammed}.~As a consequence, the question of how a random dynamical system can be obtained from an SPDE is not trivial, since solutions are only defined almost surely which is insufficient for the cocycle property. In particular,  there exist exceptional sets which depend on the initial condition, and if more than countably many exceptional sets occur it is unclear how  the random dynamical system can be defined.
This problem was fully solved only under restrictive assumptions on the structure of the noise.~More precisely, for SPDEs with additive or linear multiplicative noise, there are standard transformations which reduce these SPDEs in PDEs with random coefficients. Since random PDEs can be solved pathwise, the generation of the random dynamical system is straightforward.\\

\medskip

Before we recall the notions of global and exponential random attractors we need to introduce  the class of tempered random sets. 
From now on, in this subsection and Subsection 2.2, when stating properties involving 
a random parameter, we assume, unless otherwise specified, that they hold on a $(\theta_t)_{t\in\mathbb{R}}$-invariant subset of $\Omega$ of full measure, i.e. there exists a $(\theta_t)_{t\in\mathbb{R}}$-invariant subset $\Omega_0\subset\Omega$ of full measure such that the property holds for all $\omega\in\Omega_0.$ To simplify notations we denote $\Omega_0$ again by $\Omega.$

\begin{definition}
 A multifunction $\mathcal{B}=\{B(\omega)\}_{\omega\in\Omega}$ of nonempty closed subsets 
 $B(\omega)$ of $X$ is called a random set if 
 $$
 \omega\mapsto\inf_{y\in B(\omega)}\|x-y\|_X
 $$ 
 is a random variable for each $x\in X$.

The random set $\mathcal{B}$ is bounded (or compact) if the sets $B(\omega)\subset X$ are bounded (or compact) 
for all $\omega\in\Omega.$
\end{definition}

\begin{definition} 
A random bounded set $\{B(\omega)\}_{\omega\in\Omega}$ of $X$ is called tempered with respect to $(\theta_{t})_{t\in\mathbb{R}}$ if for all $\omega\in\Omega$ it holds that
	$$ \lim\limits_{t\to\infty}\text{e}^{-\beta t}\sup\limits_{x\in B(\theta_{-t}\omega)}\|x\|_X=0\quad \mbox{  for all } \beta>0.$$
\end{definition}

Here and in the sequel, we denote by $\mathcal{D}$ the collection of tempered random sets in $X$.


\begin{definition}\label{attractor}
Let $\varphi$ be a random dynamical system on $X$. A random set $\{\mathcal{A}(\omega)\}_{\omega\in\Omega}\in\mathcal{D}$ is called a 
	$\mathcal{D}$-random  (pullback) attractor for $\varphi$ if the following properties are satisfied:
	\begin{description}
		\item [a)] $\mathcal{A}(\omega)$ is compact for every $\omega\in\Omega$;
		\item [b)] $\{\mathcal{A}(\omega)\}_{\omega\in\Omega}$ is $\varphi$-invariant, i.e.
		$$\varphi(t,\omega,\mathcal{A}(\omega))=\mathcal{A}(\theta_{t}\omega)\mbox{ for all } t\geq 0, \omega\in\Omega;$$
		\item [c)] $\{\mathcal{A}(\omega)\}_{\omega\in\Omega}$ pullback attracts every set in $\mathcal{D}$, i.e. for every $D=\{D(\omega)\}_{\omega\in\Omega}\in\mathcal{D}$,
		$$\lim\limits_{t\to\infty}d(\varphi(t,\theta_{-t}\omega,D(\theta_{-t}\omega)),\mathcal{A}(\omega))=0, $$
		where $d$ denotes the Hausdorff semimetric in $X$,
		$d(A,B)=\sup\limits_{a\in A}\inf\limits_{b\in B}\|a-b\|_X$, for any subsets $A\subseteq X$ and $B\subseteq X$.
	\end{description}
\end{definition}

The following theorem provides a criterion for the existence of random attractors, see Theorem 4 in \cite{CrKl}. The uniqueness follows from Corollary 1 in \cite{CrKl}.

\begin{theorem}\label{thm_exist_attractor}
There exists a $\mathcal{D}$-random (pullback) attractor for $\varphi$ if and only if there exists a
compact random set that pullback attracts all random sets $D\in\mathcal{D}$.  
Moreover, the random (pullback) attractor is unique. 
\end{theorem}

One way of proving the existence of the random  attractor, that in addition implies its 
finite fractal dimension, is to show that a random exponential attractor exists. 
Exponential attractors are compact subsets of finite fractal dimension that contain the global attractor and are attracting at an exponential rate. This notion was first introduced for semigroups in the autonomous deterministic setting \cite{EdFoNiTe} and has later been extended for nonautonomous and random dynamical systems, see \cite{CarSo, CaSo} and the references therein.   

Here, we consider so-called nonautonomous random exponential attractors, see 
\cite{CaSo}. While random exponential attractors in the strict sense are positively $\varphi$-invariant, nonautonomous 
random exponential attractors are positively $\varphi$-invariant in the weaker, nonautonomous sense. To construct exponential attractors 
for time-continuous random dynamical systems 
that are positively $\varphi$-invariant
typically requires the H\"older continuity in time of the cocycle
which is a restrictive assumption. 
However, if we relax the invariance property and consider nonautonomous random exponential attractors instead, only the Lipschitz continuity of the cocycle in space is needed. In fact, the construction can be essentially simplified, we obtain better bounds for the fractal dimension and the assumption of H\"older continuity in time can be omitted, see \cite{CaSo}. Even though we could prove the H\"older continuity in time for the cocycle for our particular problem, we omit it since it has no added value for our main results and would lead to weaker bounds for the fractal dimension.

\begin{definition}\label{expattractor}
	A nonautonomous tempered random set $\{\mathcal{M}(t,\omega)\}_{t\in\R, \omega\in\Omega}$ is called a nonautonomous
	$\mathcal{D}$-random  (pullback) exponential attractor for $\varphi$ if there exists $\tilde t>0$ such that $\mathcal{M}(t+\tilde t,\omega)=\mathcal{M}(t,\omega)$ for all $t\in\R,\omega\in\Omega,$ and
	the following properties are satisfied:
	\begin{description}
		\item [a)] $\mathcal{M}(t,\omega)$ is compact for every $t\in\R,\omega\in\Omega$;
		\item [b)] $\{\mathcal{M}(t,\omega)\}_{t\in\R,\omega\in\Omega}$ is positively $\varphi$-invariant in the nonautonomous sense, i.e.
		$$\varphi(s,\omega,\mathcal{M}(t,\omega))\subseteq\mathcal{M}(s+t,\theta_{s}\omega)\quad\mbox{ for all } s\geq 0, t\in\R, \omega\in\Omega;$$
		\item [c)] $\{\mathcal{M}(t,\omega)\}_{t\in\R, \omega\in\Omega}$ is pullback $\mathcal{D}$-attracting at an exponential rate, i.e. there exists $\alpha>0$ such that
		$$\lim\limits_{s\to\infty}\text{e}^{\alpha s}d(\varphi(s,\theta_{-s}\omega,D(\theta_{-s}\omega)),\mathcal{M}(t,\omega))=0\quad\text{for all } D\in \mathcal{D}, t\in\R,\omega\in\Omega;
		$$
		\item [d)] the fractal dimension of $\{\mathcal{M}(t,\omega)\}_{t\in\R, \omega\in\Omega}$ is 
		finite, i.e. there exists a random variable $k(\omega)\geq 0$ such that 
		$$\sup_{t\in\R}\text{dim}_f(\mathcal{M}(t,\omega))\leq k(\omega)<\infty\quad\text{for all }  \omega\in \Omega.
		$$		
	\end{description}
\end{definition}

We recall that the fractal dimension of a 
precompact subset $M\subset X$ is defined as 
$$
\text{dim}_f(M)=\limsup_{\varepsilon\to 0}\log_{\frac{1}{\varepsilon}}(N_\varepsilon(M)),
$$
where $N_\varepsilon(M)$
denotes the minimal number of $\varepsilon$-balls  in $X$
with centers in $M$ needed to cover the set $M$.
\medskip

By Theorem \ref{thm_exist_attractor} the existence of a nonautonomous random exponential attractor 
immediately implies that the (global) random attractor exists. Moreover, the global random
attractor is contained in the random exponential attractor and hence, its fractal dimension is finite. 
\medskip 

Existence proofs for global and exponential random attractors are typically based on 
the existence of a pullback $\mathcal{D}$-absorbing set for $\varphi$. 

\begin{definition}\label{absorbing}
	A set $\{B(\omega)\}_{\omega\in\Omega}\in\mathcal{D}$ is called random pullback $\mathcal{D}$-absorbing for $\varphi$ if for every $D=\{D(\omega)\}_{\omega\in\Omega}\in\mathcal{D}$ and  $\omega\in\Omega$, there exists a random  time $T_{D}(\omega)\geq 0$ such that
	$$ \varphi(t,\theta_{-t}\omega,D(\theta_{-t}\omega))\subseteq B(\omega)\quad \mbox{ for all } t\geq T_{D}(\omega).$$
\end{definition}

The following condition is convenient to show the existence of an absorbing set. 
Namely, if for every $x\in D(\theta_{-t}\omega)$, $D\in\mathcal{D}$ and $\omega\in\Omega$ it 
holds that 
\begin{align}\label{criterion}
\limsup\limits_{t\to\infty} \|\varphi(t,\theta_{-t}\omega,x)\|_X\leq \rho(\omega),
\end{align}
where $\rho(\omega)>0$ for every $\omega\in\Omega$, then the ball $B(\omega):=B(0,\rho(\omega)+\delta)$ centered in $0$ with radius $\rho(\omega)+\delta$ for some constant $\delta>0,$ is a random absorbing set.
For further details and applications see~\cite{schmalfuss, BessaihGarridoSchmalfuss}.
\medskip

Instead of considering random exponential attractors which is typically more involved and 
requires to verify additional properties of the cocycle, 
the existence of random attractors is frequently shown using the following result, 
see Theorem 2.1 in \cite{schmalfuss}. 

\begin{theorem}\label{attractorbounded}
	Let $\varphi$ be a continuous random dynamical system on $X$ over $(\Omega,\mathcal{F},\mathbb{P},(\theta_{t})_{t\in\mathbb{R}})$. 
	Suppose that $\{B(\omega)\}_{\omega\in\Omega}$ is a compact random absorbing set for $\varphi$ in $\mathcal{D}$. 
	Then $\varphi$ has a unique $\mathcal{D}$-random attractor $\{\mathcal{A}(\omega)\}_{\omega\in\Omega}$ which is given by
	$$\mathcal{A}(\omega)=\bigcap\limits_{\tau\geq 0}\overline{\bigcup\limits_{t\geq\tau}\varphi(t,\theta_{-t}\omega,B(\theta_{-t}\omega))}. $$
\end{theorem}

We could apply Theorem \ref{attractorbounded} to prove the existence of a random attractor 
 for our particular
problem. However, showing that a nonautonomous random exponential attractor exists does not only imply the existence of the random attractor, but also  its finite fractal dimension. Moreover,  
it turns out to be even simpler in our case than 
applying Theorem \ref{attractorbounded}. To this end we use an existence result for 
random exponential attractors obtained in \cite{CaSo} that we recall in the next subsection.

\subsection{An existence result for random exponential attractors}\label{subsec_expattr}

The existence result for random pullback exponential attractors 
is based on an auxiliary normed space that is compactly embedded into the phase space 
and the entropy properties of this embedding. 
We recall 
some notions and results that we will need in the sequel, see also \cite{CaSo,CarSo,CarSo2}.
\medskip

The (Kolmogorov) $\varepsilon$-entropy of a precompact subset $M$ of a Banach space $X$
is defined as 
$$
\mathcal H_\varepsilon^X(M)=\log_2(N_\varepsilon^X(M)),
$$
where $N_\varepsilon^X(M)$
denotes the minimal number of $\varepsilon$-balls in $X$
with centres in $M$ needed to cover the set $M$.
It was first introduced by Kolmogorov and Tihomirov in \cite{KoTi}.
The order of growth of $\mathcal H_\varepsilon^X(M)$ as $\varepsilon$ tends to zero is a measure for the massiveness 
of the set $M$ in $X$, even if its fractal dimension is infinite.


If $X$ and $Y$ are Banach spaces such that the embedding $Y\hookrightarrow X$
is compact, we use the notation
$$
\mathcal H_\varepsilon(Y;X)=\mathcal H_\varepsilon^X(B^Y(0,1)),
$$
where $B^Y(0,1)$ denotes the closed unit ball in $Y$.

\begin{remark}\label{rmk_entropy2}
The $\varepsilon$-entropy is related to the entropy numbers $\hat e_k$ for the embedding 
$Y\hookrightarrow X,$ which are defined by
$$
\hat e_k=\inf\Big\{ \varepsilon>0 : B^Y(0,1)\subset \bigcup_{j=1}^{2^{k-1}}B^X(x_j,\varepsilon),\ x_j\in X, \ j=1,\dots,2^{k-1}\Big\},
$$
$k\in\N.$
If the embedding  is compact, then $\hat e_k$ is finite for all $k\in\N$.
For certain function spaces the entropy numbers can explicitly be estimated (see \cite{EdTr}).
For instance, if $D\subset \R^n$ is a smooth bounded domain, then the embedding of the Sobolev 
spaces 
$$
W^{l_1,p_1}(D)\hookrightarrow W^{l_2,p_2}(D),\qquad l_1,l_2\in\R, \ p_1,p_2\in(1,\infty),
$$
is compact if $l_1>l_2$ and $\frac{l_1}{n} - \frac{1}{p_1} > \frac{l_2}{n}-\frac{1}{p_2}.$
Moreover, the entropy numbers grow polynomially, namely,  
$$
\hat e_k \simeq k^{-\frac{l_1-l_2}{n}}
$$
(see Theorem 2, Section 3.3.3 in \cite{EdTr}),
and consequently, 
$$
\mathcal H_\varepsilon(W^{l_1,p_1}(D);W^{l_2,p_2}(D))\leq c
\varepsilon^{-\frac{n}{l_1-l_2}},
$$
for some constant $c>0$.
Here, we write $f\simeq g,$ if there exist positive constants 
$c_1$ and $c_2$ such that 
$$
c_1f\leq g \leq c_2f.
$$
\end{remark}

The following existence result for nonautonomous random pullback exponential attractors is a 
special case of the main result in \cite{CaSo}. In fact, we formulate a simplified 
version that  suffices for the parabolic stochastic  evolution problem  we consider.
In particular, we assume that the cocycle is uniformly Lipschitz continuous and satisfies the smoothing property with a constant 
that is independent of $\omega$. 
More generally, one can allow that the constants depend on the random parameter 
$\omega$ and that the cocycle is asymptotically compact, i.e. it is the 
sum of a mapping satisfying the smoothing property and a contraction.  

\begin{theorem}\label{thm_rexpa}
	Let $\varphi$ be a  random dynamical system in a separable Banach space $X$ and let $\mathcal{D}$ 
	denote the universe of tempered random sets. 
	Moreover, we assume that the following properties hold for  all $\omega\in\Omega$:
	\begin{itemize}
	\item[$(H_1)$] Compact embedding: 
	There exists  another separable Banach space $Y$ that is compactly and densely embedded into $X$.
	\item[$(H_2)$] Random pullback absorbing set: 
	There exists a random closed set $B\in\mathcal{D}$ that is pullback $\mathcal{D}$-absorbing, and 
	the absorbing time corresponding to a random set $D\in\mathcal{D}$ satisfies $T_{D,\theta_{-t}\omega}\leq T_{D,\omega}$ for all $t\geq 0$.
	\item[$(H_3)$] Smoothing property:
	There exists $\tilde t>T_{B,\omega}$ and a constant $\kappa>0$ such that 
	$$
	\|\varphi(\tilde t,\omega,u)-\varphi(\tilde t,\omega,v)\|_Y\leq \kappa\|u-v\|_X\qquad \forall u,v\in B(\omega).
	$$
	\item[$(H_4)$] Lipschitz continuity: There exists 
	a constant $L_\varphi>0$ such that 
	$$
	\|\varphi(s,\omega,u)-\varphi(s,\omega,v)\|_{X}\leq L_\varphi\|u-v\|_{X}\qquad \forall s\in[0,\tilde{t}],\ u,v\in B(\omega).
	$$
	\end{itemize}
	Then, for every $\nu\in(0,\frac{1}{2})$ 
	there exists a nonautonomous random pullback exponential attractor, and its fractal dimension is uniformly bounded by  
	$$
	\textnormal{dim}_f(\mathcal{M}^\nu(t,\omega))\leq 
	\log_{\frac{1}{2\nu}}\left(N_{\frac{\nu}{\kappa}}^X(B^Y(0,1))\right) \qquad \forall t\in\R,\ \omega\in\Omega.
	$$
\end{theorem}

\subsection{Pathwise mild solutions for parabolic SPDEs}\label{subs_spdes}

Let $\Delta:=\{(s,t)\in\R^2:\mbox{ } s\leq t\}$, $X$ be a separable, reflexive, type 2 Banach space and 
$(\overline\Omega,\overline{\mathcal{F}}, \overline{\mathbb{P}})$ be a probability space.
Similarly to~\cite{PronkVeraar}, we consider nonautonomous SPDEs of the form
\begin{align}\label{eqinit}
du(t) &= A(t,\overline{\omega}) u(t) ~\txtd t + F(u(t)) ~\txtd t + \sigma(t,u(t))~\txtd W_{t},&& t>s,\\
u(s)&=u_{0} \in X, && s\in\R,\nonumber
\end{align}
where $A=\{A(t,\overline\omega)\}_{t\in\R,\overline{\omega}\in\overline{\Omega}}$ is a family of time-dependent random differential operators. Intuitively, this means that the differential operator depends on a stochastic processes, in a meaningful way which will be specified later.
\medskip

We aim to investigate the longtime behavior of~\eqref{eqinit} using a random dynamical systems approach. First, we recall sufficient conditions that ensure that the family $A$ generates 
a parabolic stochastic evolution system, see \cite{PronkVeraar}.
In particular, 
 we make the following assumptions concerning measurability, sectoriality and 
H\"older continuity of the operators.
\medskip

\newpage
\textbf{Assumptions 1}
\begin{itemize}
\item [(A0)] We assume that the operators are closed, densely defined and have a common domain, 
$\mathcal{D}_A:=D(A(t,\overline\omega))$ for all $t\in\R$, $\overline{\omega}\in\overline{\Omega}$. 
	\item [(A1)] The mapping $A:\R\times\overline{\Omega}\to\mathcal{L}(\mathcal{D}_A,X)$ is strongly measurable and adapted.
	\item [(A2)] There exists $\vartheta\in(\pi,\frac{\pi}{2})$ and $M>0$ such that $\Sigma_\vartheta:=\{\mu\in\C : |\text{arg }\mu|<\vartheta\}\subset \rho(A(t,\overline{\omega}))$ and 
	$$
	\|R(\mu,A(t,\overline{\omega}))\|_{\mathcal{L}(X)}\leq \frac{M}{|\mu|+1}\qquad \text{for all}\ \mu\in \Sigma_\vartheta\cup\{0\}, t\in\R, \overline{\omega}\in\overline{\Omega}.
	$$
	
	
	\item [(A3)] There exists $\nu\in(0,1]$ and a mapping $C:\overline{\Omega}\to X$ such that
	\begin{align}\label{kt}
	\|A(t,\overline{\omega}) - A(s,\overline{\omega})\|_{\mathcal{L}(\mathcal{D}_A,X)} \leq C(\overline\omega) |t-s|^{\nu}\qquad \text{for all}\ s,t\in\R,~ \overline{\omega}\in\overline{\Omega},	
	\end{align}
	where we assume that $C(\overline\omega)$ is uniformly bounded with respect to $\overline\omega$, see~\cite{PronkVeraar}. 
\end{itemize}

Assumptions (A2) and (A3) are referred to in the literature as the Kato-Tanabe assumptions, compare~\cite{Amann2}, p.~55, or~\cite{pazy}, p.~150, and are common in the context of nonautonomous evolution equations. 
Since the constants in (A2) and (A3) are uniformly bounded 
w.r.t. $\overline{\omega}$, all constants arising in the estimates below do not dependent on 
$\overline{\omega}$.
\\

In the sequel, we denote by $X_\eta$, $\eta\in(-1,1]$, the fractional power spaces $D((-A(t,\overline{\omega}))^\eta)$
endowed with the norm $\|x\|_{X_\eta}=\|(-A(t,\overline\omega))^\eta x\|_X$ for $t\in\mathbb{R}$,  $\overline{\omega}\in\overline\Omega$ and $x\in X_\eta$.
\medskip

\textbf{Assumptions 2}
\begin{itemize}
	\item[(AC)] We assume that the operators $A(t,\overline{\omega}), t\in\R,\overline{\omega}\in\overline{\Omega},$ have a compact inverse. This implies that the embeddings 
$X_\eta\hookrightarrow X$, $\eta\in(0,1]$, are compact.
	\item [(U)] The evolution family is uniformly exponentially stable, i.e. there exist constants $\lambda>0$ and $c>0$ such that
	\begin{align}
	&\|U(t,s,\overline{\omega})\|_{\mathcal{L}(X)} \leq c e ^{-\lambda (t-s)} \quad \mbox{for all } (s,t)\in\Delta \mbox{ and } \overline{\omega}\in\overline{\Omega}.\label{omega:constant}
	\end{align}
	\item [(Drift)] The nonlinearity $F:X\to X$ is globally Lipschitz continuous, i.e.~there exists a constant $C_{F}>0$ such that
	\begin{align*}
	\|F(x) -F(y)\|_{X}\leq C_{F}\|x-y\|_{X}\quad \mbox{for all } x,y\in X \mbox{ and } \overline\omega\in\overline{\Omega}.
	\end{align*}
	 This implies a linear growth condition on $F$. Namely, there exist a positive constant $\overline{C}_{F}$ such that
	\begin{align}\label{constants:F}
	\|F(x)\|_{X}\leq \overline{C}_{F} + C_{F}\|x\|_{X}\quad \mbox{ for all } x\in X \mbox{ and } \overline\omega\in\overline{\Omega}.
	\end{align}
Furthermore, we assume that $\lambda - c C_{F}>0$.
	\item [(Noise)] We assume that $W(t)$ is a two-sided Wiener process with 
	values in $X_{\beta}$, $\beta \in(0,1]$. Furthermore, we set $\sigma(t,u):=\sigma>0$.
\end{itemize}

Based on the Assumptions 1, by applying~\cite[Thm.~2.3]{Acq} pointwise in $\overline{\omega}\in\overline{\Omega}$ we obtain the following theorem,
see \cite[Theorem 2.2]{PronkVeraar}. The measurability was shown in \cite[Proposition 2.4]{PronkVeraar}.
Before we state the result we recall the definition of strong measurability of random operators. 
\begin{definition}
Let $X_1$ and $X_2$ be two separable Banach spaces. A random operator $L:\overline{\Omega}\times X_1\to X_2$ is called strongly measurable if the mapping $\overline{\omega}\mapsto L(\overline{\omega})x$, $\bar\omega\in\overline{\Omega}$, is a random variable on $X_2$ for every $x\in X_1$.
\end{definition}

\begin{theorem}\label{thm_generation}
	There exists a unique parabolic evolution system $U:\Delta\times\overline{\Omega}\rightarrow\mathcal{L}(X)$ with the following properties:
	
	\begin{enumerate}
		\item [$1)$]  $U(t,t,\overline{\omega})=\mbox{Id}$ \ for all $t\geq 0$, $\overline{\omega}\in\overline{\Omega}$.
		\item [$2)$] We have
		\begin{equation}\label{cozyklus}
		U(t,s,\overline{\omega})U(s,r,\overline{\omega})=U(t,r,\overline{\omega})
		\end{equation}
		 for all $0\leq r\leq s\leq t$, $\overline{\omega}\in\overline{\Omega}$.
		\item [$3)$] The mapping $U(\cdot,\cdot,\overline\omega)$ is strongly continuous for all $\overline{\omega}\in\overline{\Omega}$.
		\item [$4)$] For $s<t$ the following identity 
		holds pointwise in $\overline{\Omega}$ $$\frac{d}{dt}U(t,s,\overline{\omega})=A(t,\overline{\omega})U(t,s,\overline{\omega}).
		$$
		\item [$5)$] The evolution system $U:\Delta\times\overline{\Omega}\to\mathcal{L}(X)$ is strongly 
		measurable in the uniform operator topology. Moreover, for every $t\geq s$, 
		the mapping $\overline{\omega}\mapsto U(t,s,\overline{\omega})\in \mathcal{L}(X)$ is strongly $\mathcal{F}_t$-measurable in the uniform operator topology. 	
	\end{enumerate}
\end{theorem}

To prove the existence of random attractors we need additional smoothing properties of the 
evolution system. 
The following properties and estimates were shown in Lemma 2.6 and Lemma 2.7 in \cite{PronkVeraar}.
The exponential decay is a consequence of our assumption $(U)$.

\begin{lemma}\label{lem_estimates}
We assume that the family of 
adjoint operators $A^*(t,\overline{\omega})$ satisfies $(A3)$ with exponent $\nu^*>0$. Then, for every $t>0$, the mapping $s\mapsto U(t,s,\overline{\omega})$ belongs to $C^1([0,t);\mathcal{L}(X))$, and for all $x\in\mathcal{D}_A$ one has 
$$\frac{d}{ds}U(t,s,\overline{\omega})x=-U(t,s,\overline{\omega})A(s,\overline{\omega})x.
$$ 

Moreover, for
$\alpha\in [0,1]$ and $\eta\in(0,1)$ there exist positive 
constants $\widetilde{C}_\alpha, \widetilde{C}_{\alpha,\eta}$ such that the following estimates hold for $t>s$ and $\bar\omega\in\overline{\Omega}$:
		\begin{align*}
		\|(-A(t,\overline\omega))^\alpha U(t,s,\overline\omega)x\|_X &\leq \widetilde{C}_\alpha \frac{\text{e}^{-\lambda(t-s)}}{(t-s)^{\alpha}}\|x\|_X, 
		&&x\in X;\\
		\|U(t,s,\overline\omega)(-A(s,\overline\omega))^\alpha x\|_X &\leq \widetilde{C}_\alpha\frac{e^{-\lambda(t-s)}}{(t-s)^{\alpha}}\|x\|_X,
		&& x\in X_\alpha;\\
		\|(-A(t,\overline\omega))^{-\alpha} U(t,s,\overline\omega) (-A(s,\overline\omega))^\eta x\|_X& \leq \widetilde{C}_{\alpha,\eta} \frac{\text{e}^{-\lambda(t-s)}}{(t-s)^{\eta -\alpha}}\|x\|_X, &&x\in X_\eta.		
		\end{align*}
\end{lemma}

To shorten notations, in the sequel we omit the $\overline{\omega}$-dependence of $A$ and $U$ if there is no danger of confusion. 
The classical mild formulation of the SPDE ~\eqref{eqinit} is 
\begin{align}\label{pms}
u(t) = U(t,0)u_{0} + \int\limits_{0}^{t} U(t,s)F(u(s)) ~\txtd s + \sigma \int\limits_{0}^{t} U(t,s)~\txtd W(s).
\end{align}
However, the It\^{o}-integral is not well-defined since the mapping $\overline{\omega}\mapsto U(t,s,\overline{\omega})$ 
is, in general, only $\mathcal{F}_{t}$-measurable and not $\mathcal{F}_{s}$-measurable, see
\cite[Prop.~2.4]{PronkVeraar}. To overcome this problem Pronk and Veraar introduced in \cite{PronkVeraar} the concept of pathwise mild solutions. In our particular case, this notion leads to the integral representation
\begin{align}\label{sol:1}
\begin{split}
u(t)=&\  U (t,0) u_{0} + \sigma U(t,0) W(t) + \int\limits_{0}^{t} U(t,s)F(u(s))~\txtd s\\
&\ -\sigma \int\limits_{0}^{t}U(t,s)A(s) (W(t) -W(s)) ~\txtd s.
\end{split}
\end{align}
The formula is motivated by formally applying integration by parts for the stochastic integral, and, as shown in~\cite{PronkVeraar}, it indeed yields a pathwise representation for the solution.\\

Our aim is to show the existence of random attractors for SPDEs using this concept of pathwise mild solutions.
It allows us to study random attractors without transforming the SPDE into a random PDE, as it is typically done.
\begin{remark}
	We emphasize that the concept of pathwise mild solutions also applies if $\sigma$ is not constant, see~\cite[Sec.5]{PronkVeraar}. In this case, the solution of~\eqref{eqinit} is given by
	\begin{align}
	u(t) = U(t,0)u_{0} &+ U(t,0) \int\limits_{0}^{t}\sigma(s,u(s))~\txtd W(s) + \int\limits_{0}^{t} U(t,s)F(u(s))~\txtd s\label{one}\\& - \int\limits_{0}^{t} U(t,s)A(s)\int\limits_{s}^{t}\sigma(\tau,u(\tau))~\txtd W(\tau)~\txtd s. \label{two}
	\end{align}
	However, it is not possible to obtain a random dynamical system in this case, due to the presence of the stochastic integrals in~\eqref{one} and~\eqref{two} which are not defined in a  pathwise sense. Consequently, this representation formula does not hold for {\em every} $\overline{\omega}\in\overline{\Omega}$.
	We aim to investigate this issue in a future work.
\end{remark}

Recalling that $W$ is an $X_{\beta}$-valued Wiener process, 
we introduce the canonical probability space
\begin{align}\label{mds}
\Omega:=(C_{0}(\mathbb{R};X_{\beta}),\mathcal{B}(C_{0}(\mathbb{R};X_{\beta})),\mathbb{P})
\end{align}
and identify $W(t,\omega)=:\omega(t)$, for $\omega\in\Omega$. Moreover, together with the Wiener shift,
\begin{align*}
\theta_{t}\omega(s)=\omega(t+s)-\omega(t), \quad\omega\in\Omega, \ s,t\in\mathbb{R},
\end{align*}
we obtain, analogously as in Subsection \ref{subsec_rds}, the ergodic metric dynamical system $(\Omega,\mathcal{F},\mathbb{P},(\theta_{t})_{t\in\mathbb{R}})$.\\

In the following, $(\Omega,\mathcal{F},\mathbb{P})$ always denotes the probability space \eqref{mds}.

%
%
%
%

\section{Random attractors for nonautonomous random SPDEs}\label{ass:attr}

\subsection{Random dynamical system and absorbing set}
	Since we consider SPDEs with  time-dependent random differential operators, we need 
	to impose additional structural assumptions  
	in order to use the framework of random dynamical systems, see~\cite{CDLS,lns}. 
	\medskip
	
\textbf{Assumptions 3}
	
	\begin{itemize}
	\item[(RDS)]
	We assume that the generators depend on $t$ and $\omega$ in the following way: 
	\begin{align}\label{struc_ass}
	A(t,\omega)=A(\theta_{t}\omega)\quad \mbox{for all } t\in\R \mbox{ and }\omega\in\Omega.
	\end{align}
	\end{itemize}
	
This assumption is needed to obtain the cocycle 
	property.	
	In this case, the group property of the metric dynamical system implies hat 
	\begin{align*}
	A(\theta_{s}\theta_{t-s}\omega)=A(\theta_{t}\omega)\quad \mbox{for all  } t, s\in\R\mbox{ and } \omega\in\Omega.
	\end{align*}

Moreover, one can easily show that  $A(\theta_t\omega)$ generates a random dynamical system, i.e. the solution operator corresponding to the linear evolution equation
\begin{align*}
\txtd u(t)& = A(\theta_t\omega) u(t)~\txtd t\\
u(0)&=u_{0}\in X,
\end{align*}
forms a random dynamical system.
\medskip

From now on, we always assume that Assumptions 1, 2 and 3 hold and that the family of adjoint operators 
$A^*$ satisfies $(A3)$ with $\nu^*\in(0,1].$

\begin{theorem} Let $U:\Delta\times\Omega\to\mathcal{L}(X)$ be the evolution operator generated by $A(\theta_t\omega)$. Then, $\widetilde{U}:\mathbb{R}^{+}\times\Omega\times X\to X$ defined as
	\begin{align}
	\widetilde{U}(t,\omega):=U(t,0,\omega),\quad t\geq 0,
	\end{align}
	is a random dynamical system.
\end{theorem}

\begin{proof}
	The cocycle property immediately follows from~\eqref{cozyklus}.	In fact, let $t,s\geq 0$. Then, ~\eqref{cozyklus} implies that
	\begin{align*}
	U(t+s,0,\omega)=U(t+s,s,\omega)U(s,0,\omega).
	\end{align*}
	Moreover, we observe that  $U(t+s,s,\omega)=U(t,0,\theta_{s}\omega)$ since $A(\theta_{t}\omega)=A(\theta_{s}\theta_{t-s}\omega)$.
	Intuitively, this means that starting at time $s$ on the $\omega$-fiber of the noise and letting time $t>0$ pass, leads to the same state as starting at time zero on the shifted $\theta_s$-fiber of the noise and letting the system evolve for time $t$.
	At the level of random generators, $U(t+s,s,\omega)$ is obtained from $A(\theta_t\omega)$ which is the same as $A(\theta_s \theta_{t-s}\omega)$ due to the properties of the metric dynamical system. Therefore, the cocycle property
	\begin{align}\label{ex:c}
	\widetilde{U}(t+s,\omega)=\widetilde{U}(t,\theta_{s}\omega) \widetilde{U}(s,\omega).
	\end{align}  is satisfied.
	The measurability of $\tilde{U}$ follows from Theorem \ref{thm_generation}, Property 5).
\end{proof}\\


\begin{remark}
We obtain the measurability of $\widetilde{U}$ directly from the results in \cite{PronkVeraar}.
Alternatively, one can show the measurability of $\tilde{U}$ as in~\cite[Lem.~14]{lns} using Yosida approximations of $A(\omega)$. Here, one argues that the evolution operators corresponding to these approximations are strongly measurable and then passes to the limit. These arguments exploit the structural assumption \eqref{struc_ass}.
The proof of the measurability in \cite{PronkVeraar} is more involved and holds under more general assumptions.
\end{remark}

We give a standard example of a random nonautonomous generator and its corresponding evolution operator.

\begin{example}\label{ex}
	A simple example for an operator that satisfies our assumptions is a random perturbation of a uniformly elliptic operator $A$ (in a smooth bounded domain with homogeneous Dirichlet boundary conditions) by a real-valued Ornstein-Uhlenbeck process, which is the stationary solution of the Langevin equation
	\begin{align*}
	\txtd z(t) = -\mu z(t)~dt + ~\txtd \overline{W}(t).
	\end{align*}
	Here, $\mu>0$ and $\overline{W}$ is a two-sided real-valued Brownian motion. We denote by $(\overline{\Omega},\overline{\mathcal{F}},\overline{\mathbb{P}})$ its associated canonical probability space and make the identification $\overline{W}(t,\overline{\omega}):=\overline{\omega}(t)$ for $\overline\omega\in\overline\Omega$.
	Then, we have that
	\begin{align*}
	z(\theta_{t}\overline\omega) =\int\limits_{-\infty}^{t} \text{e}^{-\mu(t-s)}~\txtd \overline{\omega}(s) =\int\limits_{-\infty}^{0} \text{e}^{\mu s}~\txtd\theta_{t}\overline\omega(s).
	\end{align*}
	In this case, the parabolic evolution operator generated by $A+z(\theta_t\overline\omega)$ is \begin{align*}
	\widetilde{U}(t,\omega) :=T(t) \text{e}^{\int\limits_{0}^{t}z(\theta_{\tau}\overline\omega)~\txtd\tau},
	\end{align*}
	where $(T(t))_{t\geq 0}$ is the analytic $C_{0}$-semigroup generated by $A$. We have 
	\begin{align*}
	\widetilde{U}(t,\overline\omega) =  \underbrace{T(t-s)\text{e}^{\int\limits_{s}^{t}z(\theta_{\tau}\overline\omega)~\txtd\tau }}_{U(t,s,\overline\omega)} \underbrace{T(s)\text{e}^{\int\limits_{0}^{s}z(\theta_{\tau}\overline\omega)~\txtd\tau}}_{U(s,0,\overline\omega)},
	\end{align*}
	and consequently, $\widetilde{U}(t-s,\theta_s\overline\omega)=T(t-s)\text{e}^{\int\limits_{0}^{t-s}z(\theta_{\tau+s}\overline\omega)} ~\txtd\tau$.\\
	
\end{example}

This simple example illustrates that the formalism we introduced above is meaningful.
Further examples of random time-dependent generators are provided in Section~\ref{examples}. For additional applications we refer to \cite{PronkVeraar}, and to 
 \cite{CDLS,lns} in the context of random dynamical systems.\\

We now prove the existence of random attractors for SPDEs of the form
\begin{align}\label{equation}
 \begin{cases}
\txtd u(t) = A(\theta_{t}\omega) u(t) ~\txtd t + F(u(t)) ~\txtd t + \sigma ~\txtd\omega(t)\\
u(0)=u_{0}
\end{cases}
\end{align}
using pathwise mild solutions as defined in~\eqref{sol:1}.

\begin{remark}
\begin{itemize}
\item 

	We emphasize that the SPDE~\eqref{equation} cannot be transformed into a PDE with random coefficients using the stationary Ornstein-Uhlenbeck process, since the convolution
	$$\int\limits_{0}^{t} U(t,s,\omega)~\txtd \omega(s) $$ is not defined
	and one has to make sense of it using the integration by parts formula
	\begin{equation}\label{ou}
	\omega(t) + \int\limits_{0}^{t} U(t,s,\omega)A(\theta_s\omega) \omega(s)~\txtd s = U(t,0,\omega)\omega(t) - \int\limits_{0}^{t} U(t,s,\omega)A(\theta_s\omega)(\omega(t) -\omega(s))~\txtd s.
	\end{equation}
A transformation based on a strictly stationary solution of~\eqref{problem} has been introduced in \cite{Gess2} using a different approach.
	\item 
	Another approach would be to subtract the noise, i.e.~to introduce the change of variables  $v:=u-\sigma \omega$. This would {\em formally} lead to the random PDE 
	\begin{align}\label{s_noise}
	\txtd v(t) &= A(\theta_{t}\omega) (v(t) +\sigma\omega(t)) ~\txtd t + F (v(t) +\omega(t) )~\txtd t\nonumber\\
	& = A(\theta_{t}\omega) v(t)~\txtd t + \sigma A(\theta_{t}\omega)\omega(t)~\txtd t +  F (v(t) +\omega(t) )~\txtd t.
	\end{align}
	The mild solution of~\eqref{s_noise} would be given by 
	$$v(t) =U(t,0,\omega) v_0 + \sigma\int\limits_{0}^{t} U(t,s,\omega)A(\theta_s\omega)\omega(s)~\txtd s + \int\limits_{0}^{t} U(t,s,\omega)F(v(s)+\sigma\omega(s))~\txtd s. 
	$$
	However, we would need to justify that this mild solution is well-defined, and the noise also interacts with the nonlinear term. 
	In order to avoid these difficulties we work with pathwise mild solutions.
\end{itemize}
\end{remark}

Using~\eqref{ou}, the   
 representation formula of a solution for~\eqref{equation} reads as 
\begin{align}
u(t)&= U(t,0)u_{0} + \sigma U(t,0)\omega(t) + \int\limits_{0}^{t}U(t,s)F(u(s))~\txtd s - \sigma \int\limits_{0}^{t} U(t,s)A(\theta_{s}\omega) (\omega(t) -\omega(s)) ~\txtd s\nonumber\\
& =  U(t,0)u_{0} + \sigma U(t,0)\omega(t) +\int\limits_{0}^{t}U(t,s)F(u(s))~\txtd s- \sigma \int\limits_{0}^{t} U(t,s)A(\theta_{s}\omega) \theta_{s}\omega(t-s) ~\txtd s\nonumber\\
&= \widetilde{U}(t,\omega)u_{0} + \sigma\widetilde{U}(t,\omega)\omega(t) + \int\limits_{0}^{t} \widetilde{U}(t-s,\theta_{s}\omega)F(u(s))~\txtd s\label{for:absorbing}\\
&\ \  -\sigma\int\limits_{0}^{t}\widetilde{U}(t-s,\theta_{s}\omega)A(\theta_{s}\omega) \theta_{s}\omega(t-s)~\txtd s\nonumber\\
& = \widetilde{U}(t,\omega)u_{0} +\int\limits_{0}^{t} \widetilde{U}(t-s,\theta_{s}\omega) F(u(s)) ~\txtd s +\sigma\omega(t) +\sigma\int\limits_{0}^{t} \widetilde{U}(t-s,\theta_{s}\omega)A(\theta_{s}\omega) \omega(s)~\txtd s.\label{cocycle:p}
\end{align}
Here, we used in the last line that
\begin{align*}
\int\limits_{0}^{t} U(t,s,\omega)A(s,\omega)~\txtd s= \int\limits_{0}^{t}\widetilde{U}(t-s,\theta_{s}\omega) A(\theta_{s}\omega)~\txtd s = -U(t,t,\omega) + U(t,0,\omega)=\widetilde{U}(t,\omega)- \mbox{Id},
\end{align*}
since
$$ \frac{\partial}{\partial s} U(t,s,\omega)=-U(t,s,\omega)A(s,\omega).$$

\begin{remark} We emphasize that the pathwise mild solution concept is applicable also under weaker assumptions on the noise, for instance if $\omega$ takes values in  a suitable extrapolation space~\cite[Section~3.1]{PortalVeraar}.
Moreover, the formal computations made in~\eqref{cocycle:p} can be justified even if $\omega\not\in\mathcal{D}_A$.

In fact, according to~\cite[Theorem 4.9]{PronkVeraar} we know that the pathwise mild solution is equivalent to the weak solution of~\eqref{equation}. For simplicity we test the linear part (i.e.~for $F\equiv 0$) of~\eqref{equation} with $x^{*}\in \mathcal{D}_{A^*}:= D((A^*(t)))$. This yields
	\begin{align}\label{weak}
	\left< u(t),x^{*}\right> = \left<U(t,0)u_0,x^{*}\right> +\sigma \left<U(t,0)\omega(t),x^{*}\right> - \sigma\int\limits_{0}^{t} \left< U(t,s) A(\theta_s\omega) (\omega(t)-\omega(s) ),x^{*}\right>~\txtd s,
	\end{align}
	where $\langle\cdot,\cdot\rangle$ denotes the dual pairing. Plugging the identity 
	\begin{align*}
	\int\limits_{0}^{t}\left<U(t,s) A(\theta_s\omega)\omega(t),x^{*}\right>~\txtd s = \left<U(t,0)\omega(t),x^{*}\right> -\left<\omega(t),x^{*}\right>,
	\end{align*}
	which holds for $\omega(\cdot)\in X$ (see \cite[Section~4.4]{PronkVeraar}),
	into ~\eqref{weak} entails
	\begin{align*}
	\left< u(t),x^{*}\right> = \left<U(t,0)u_0,x^{*}\right> + \sigma\left<\omega(t), x^{*} \right> +\sigma\int\limits_{0}^{t}\left<U(t,s)A(\theta_s\omega)\omega(s), x^{*}\right>~\txtd s.
	\end{align*}
\end{remark}

\begin{lemma} 
The solution operator corresponding to~\eqref{equation} generates a random dynamical system $\varphi:\mathbb{R}^{+}\times\Omega\times X \to X$.
\end{lemma}

\begin{proof}
We only verify the cocycle property. The continuity is straightforward and the measurability of $\varphi$ follows from the measurability of $\tilde{U}$.\\

Let $s,t\geq 0$.
Using~\eqref{cocycle:p}, we have
\begin{align*}
&\ \varphi(t+s,\omega,u_{0}) \\
=&\ \widetilde{U}(t+s,\omega)u_{0}  + \int\limits_{0}^{t+s} \widetilde{U}(t+s-r,\theta_{r}\omega) F(u(r)) 
~\txtd r+\sigma \omega(t+s) \\
&\ +\sigma \int\limits_{0}^{t+s} \widetilde{U}(t+s-r,\theta_{r}\omega) A(\theta_r\omega) \omega(r) ~\txtd r\\
 = &\ \widetilde{U}(t,\theta_{s}\omega) \widetilde{U}(s,\omega)u_{0} + \widetilde{U}(t,\theta_{s}\omega)\int\limits_{0}^{s}\widetilde{U}(s-r,\theta_{r}\omega) F(u(r))~\txtd r\\
&\  + \int\limits_{s}^{s+t} \widetilde{U}(t+s-r,\theta_{r}\omega) F(u(r)) dr+\sigma \omega(t+s)\\
&\  +\sigma\widetilde{U}(t,\theta_{s}\omega) \int\limits_{0}^{s} \widetilde{U}(s-r,\theta_{r}\omega) A(\theta_{r}\omega)\omega(r)~\txtd r 
+\sigma \int\limits_{s}^{s+t} \widetilde{U}(t+s-r,\theta_{r}\omega) A(\theta_{r}\omega)\omega(r)~\txtd r\\
 =&\  \widetilde{U}(t,\theta_{s}\omega) \Bigg[\widetilde{U}(s,\omega)u_{0} +\int\limits_{0}^{s}\widetilde{U}(s-r,\theta_{r}\omega)F(u(r))~\txtd r 
+\sigma\int\limits_{0}^{s} \widetilde{U}(s-r,\theta_{r}\omega) A(\theta_{r}\omega) \omega(r)~\txtd r  \Bigg] \\
&\  + \int\limits_{s}^{s+t} \widetilde{U}(t+s-r,\theta_{r}\omega) F(u(r)) ~\txtd r +\sigma\omega(t+s)
+\sigma\int\limits_{s}^{s+t} \widetilde{U}(t+s-r,\theta_{r}\omega) A(\theta_{r}\omega)\omega(r)~\txtd r.
\end{align*} 
Using that
\begin{align*}
&\ \int\limits_{s}^{s+t} \widetilde{U}(t+s-r,\theta_{r}\omega) A(\theta_{r}\omega)\omega(r)~\txtd r= \int\limits_{0}^{t}\widetilde{U}(t-r,\theta_{s+r}\omega)A(\theta_{s+r}\omega) \omega(r+s)~\txtd r\\
 =&\  \int\limits_{0}^{t} \widetilde{U}(t-r,\theta_{s+r}\omega)A(\theta_{s+r}\omega) (\theta_{s}\omega(r)+ \omega(s) )~\txtd r \\
=&\ \int\limits_{0}^{t} \widetilde{U}(t-r,\theta_{s+r}\omega)A(\theta_{s+r}\omega) \theta_{s}\omega(r)~\txtd r + \omega(s)(-\mbox{Id} + \widetilde{U}(t,\theta_{s}\omega)),
\end{align*}
one immediately gets
\begin{align*}
&\ \varphi(t+s,\omega,u_{0})\\
=&\  \widetilde{U}(t,\theta_{s}\omega) \Bigg[\widetilde{U}(s,\omega)u_{0} 
+\sigma\omega(s) +\int\limits_{0}^{s}\widetilde{U}(s-r,\theta_{r}\omega)F(u(r))~\txtd r+\sigma\int\limits_{0}^{s} \widetilde{U}(s-r,\theta_{r}\omega)A(\theta_{r}\omega) \omega(r)~\txtd r  \Bigg]\\
&\  +\int\limits_{0}^{t} \widetilde{U}(t-r,\theta_{s+r}\omega) F(u(r+s)) ~\txtd r
 + \sigma\int\limits_{0}^{t}\widetilde{U}(t-r,\theta_{s+r}\omega)A(\theta_{r+s}\omega)\theta_{s}\omega(r)~\txtd r +\sigma \theta_{s}\omega(t)\\
=& \  \varphi(t,\theta_{s}\omega,\varphi(s,\omega,u_{0})).
\end{align*}
This proves the statement.
	\qed
	\end{proof}\\

Next, we show the existence of an absorbing set. We recall that  $\lambda>cC_{F}$ as assumed in (Drift). Here $\lambda, c$ and $C_F$ are the constants in \eqref{omega:constant} and \eqref{constants:F}.

From now on, the properties and statements hold
for all $\omega\in\Omega_0$ where $\Omega_0\subset\Omega$ is the set of all $\omega$ that have subexponential growth. The set
$\Omega_0$ is $(\theta_t)_{t\in\R}$-invariant and has full measure, see, e.g. \cite{BessaihGarridoSchmalfuss}. 
To simplify notations, in the sequel, we will denote $\Omega_0$ again by $\Omega.$

\begin{lemma}\label{lem_abs}
 The random dynamical system $\varphi$ has a pullback absorbing set.
\end{lemma}

\begin{proof}
	We verify~\eqref{criterion}. To this end we use the estimates in Lemma \ref{lem_estimates}
	and Gronwall's Lemma. 
We observe that $\omega\in X_\beta$ implies that 
	\begin{align*}
	\|\widetilde{U}(t,\omega)\omega(t)\|_X= 
	\|\widetilde{U}(t,\omega)(-A(\omega))^{-\beta} (-A(\omega))^{\beta}\omega(t)\|_X \leq  
	\hat c \text{e}^{-\lambda t} \|\omega(t)\|_{X_{\beta}},
	\end{align*}	
	for some constant $\hat c>0.$	\\
	
	It is convenient to work with the representation formula~\eqref{for:absorbing}. Assuming that the fibre is given for $\theta_{-t}\omega$
we obtain 
\begin{align*}
\|u(t)\|_X\leq& \|U(t,0)u_{0}\|_X + \sigma \|U(t,0)\theta_{-t}\omega(t)\|_X +\int\limits_{0}^{t}\|U(t,s)F(u(s))\|_X~\txtd s\\
&+ \sigma \int\limits_{0}^{t} \|U(t,s)A(\theta_{s-t}\omega) \theta_{s-t}\omega(t-s) \|_X~\txtd s\\
\leq& c \text{e}^{-\lambda t} \|u_{0}\|_X + \sigma \hat c \text{e}^{-\lambda t}\|\omega(-t)\|_{X^\beta} +\int\limits_{0}^{t}\|U(t,s)F(u(s))\|_X~\txtd s\\
&+ \sigma \int\limits_{0}^{t} \|\widetilde{U}(t-s,\theta_{s-t}\omega)A(\theta_{s-t}\omega) \omega(s-t) \|_X~\txtd s.
\end{align*}
	For the nonlinear term the Lipschitz continuity and \eqref{omega:constant} yield
	\begin{align*}
	\Bigg\| \int\limits_{0}^{t} \widetilde{U}(t-s,\theta_{s}\omega)F(u(s))~\txtd s \Bigg\|_X\leq
	c\int\limits_{0}^{t} \text{e}^{-\lambda (t-s)} (\overline{C}_{F} + C_{F}\|u(s)\|_X)~\txtd s,
	\end{align*}
	and the generalized stochastic convolution results in 
	\begin{align*}
	&\int\limits_{0}^{t}\|\widetilde{U}(t-s,\theta_{s-t}\omega)A(\theta_{s-t}\omega) \omega(s-t)\|_X~\txtd s \\
	= &
	\int\limits_{0}^{t} \|\widetilde{U}(t-s,\theta_{s-t}\omega)(-A(\theta_{s-t}\omega))^{1-\beta}\|_{\mathcal{L}(X)}  \|(-A(\theta_{s-t}\omega))^\beta\omega(s-t)\|_X~\txtd s\\ 
\leq &\widetilde{C}_{1-\beta}\int\limits_{0}^{t}\frac{\text{e}^{-\lambda(t-s)}}{(t-s)^{1-\beta}}\|\omega(s-t)\|_{X_{\beta}}~\txtd s,
	\end{align*}	
	where $\widetilde{C}_{1-\beta}$ denotes the constant in Lemma \ref{lem_estimates}.
	Hence, combining the estimates we obtain
	\begin{align*}
	\|u(t)\|_X \leq &c \text{e}^{-\lambda t} \|u_{0}\|_X + \sigma \hat c \text{e}^{-\lambda t}
	\|\omega(-t)\|_{X_{\beta}} +c\int\limits_{0}^{t} \text{e}^{-\lambda(t-s)}(\overline{C}_{F} +C_{F}\|u(s)\|_X )~\txtd s\\
	& + \sigma \widetilde{C}_{1-\beta}\int\limits_{0}^{t} \frac{\text{e}^{-\lambda(t-s)}}{{(t-s)^{1-\beta}}} \|\omega(s-t)\|_{X_{\beta}}~\txtd s.
	\end{align*}
	Setting 
	\begin{align*}
	\gamma(t):=&c \text{e}^{-\lambda t}\|u_{0}\|_X +\sigma \hat c\text{e}^{-\lambda t}\|\omega(-t)\|_{X_{\beta}}+ 
	c\overline{C}_{F}\int\limits_{0}^{t} \text{e}^{-\lambda(t-s)} ~\txtd s + \sigma
	\widetilde{C}_{1-\beta}\int\limits_{0}^{t} \frac{\text{e}^{-\lambda (t-s) }}{(t-s)^{1-\beta}}\|\omega(s-t)\|_{X_{\beta}}~\txtd s\\
	=&c \text{e}^{-\lambda t}\|u_{0}\|_X +\sigma \hat c\text{e}^{-\lambda t}\|\omega(-t)\|_{X_{\beta}}+ 
	\frac{c\overline{C}_{F}}{\lambda} + \sigma
	\widetilde{C}_{1-\beta}\int\limits_{-t}^{0} \frac{\text{e}^{\lambda r }}{(-r)^{1-\beta}}\|\omega(r)\|_{X_{\beta}}~\txtd s,
	\end{align*}
	we can rewrite the previous inequality as
	\begin{align}\label{est:cf}
	\|u(t)\|_X \leq \gamma(t) +c C_{F} \int\limits_{0}^{t}\text{e}^{-\lambda(t-s)} \|u(s)\|_X~\txtd s.
	\end{align}
	Applying Gronwall's lemma to $\text{e}^{\lambda t}\|u(t)\|_X$ we infer that
	\begin{align*}
	\text{e}^{\lambda t }\|u(t)\|_X \leq \text{e}^{\lambda t } \gamma(t) +cC_{F} \int\limits_{0}^{t}\text{e}^{\lambda s}\gamma(s)  \text{e}^{cC_{F}(t-s)}~\txtd s,
	\end{align*}
	and multiplying with $\text{e}^{-\lambda t}$ we obtain
	\begin{align}\label{estimate:absorbing}
	\|u(t)\|_X \leq \gamma(t) +cC_{F} \int\limits_{0}^{t} \text{e}^{-(\lambda-cC_{F})(t-s)} \gamma(s) ~\txtd s.
	\end{align}
	This estimate allows us to determine the pullback absorbing set. 	 
	First, note that all terms in $\gamma$ are well-defined for the limit $t\to\infty$, due to the subexponential growth of $\|\omega(t)\|_{X_{\beta}}$, and consequently, 
	\begin{align}\label{first:contribution}
	\gamma(t) \leq e ^{-\lambda t} \Bigg(c \|u_{0}\|_X +\hat c \sigma \|\omega(-t)\|_{X_{\beta}} \Bigg) + \frac{c\overline{C}_{F}}{\lambda}
	+\sigma \widetilde{C}_{1-\beta}\int\limits_{-\infty}^{0} \frac{\text{e}^{\lambda r}}{(-r)^{1-\beta}} \|\omega(r)\|_{X_{\beta}}\txtd r<\infty.
	\end{align}		
	We now focus on the second term in~\eqref{estimate:absorbing},
	\begin{align*}
	\int\limits_{0}^{t} \text{e}^{-(\lambda-cC_{F})(t-s)}\gamma(s)~\txtd s\leq \int\limits_{0}^{t} \text{e}^{-(\lambda-cC_{F})(t-s)} \Bigg(& c\text{e}^{-\lambda s}\|u_{0}\|_X 
	+\hat c\sigma \text{e}^{-\lambda s}  \|\omega(-s)\|_{X_{\beta}}  + \frac{c\overline{C}_{F}}{\lambda}\\
	& + \sigma \widetilde{C}_{1-\beta} \int\limits_{-s}^{0} \frac{\text{e}^{\lambda r}}{(-r)^{1-\beta}} \|\omega(r)\|_{X_{\beta}} \txtd r  \Bigg) ~\txtd s.
	\end{align*}
	The first and the third term are bounded by
	\begin{align*}
	\int\limits_{0}^{t} \text{e}^{-(\lambda-cC_{F})(t-s)} \text{e}^{-\lambda s}\|u_{0}\|_X ~\txtd s\leq  \frac{\text{e}^{-(\lambda-cC_{F})t}}{cC_{F}} \|u_{0}\|_X
	\end{align*}
	and obviously
	\begin{align*}
	\frac{c\overline{C}_{F}}{\lambda}\int\limits_{0}^{t} \text{e}^{-(\lambda-cC_{F})(t-s)}~\txtd s \leq \frac{c\overline{C}_{F}}{\lambda (\lambda - cC_{F})}.
	\end{align*}
	The second one can be estimated by
	\begin{align*}
	\sigma\hat c\text{e}^{-(\lambda-cC_{F})t}\int\limits_{0}^{t} \text{e}^{-cC_{F}s}  \|\omega(-s)\|_{X_{\beta}}~\txtd s= \sigma\hat c \text{e}^{-(\lambda-cC_{F})t}\int\limits_{-t}^{0} \text{e}^{cC_{F}s}  \|\omega(s)\|_{X_{\beta}}~\txtd s.
	\end{align*}	
	Finally, for the last term, we observe that
	\begin{align*}
	&\sigma \widetilde{C}_{1-\beta}\int\limits_{0}^{t} \text{e}^{-(\lambda -cC_{F})(t-s)} \int\limits_{-s}^{0} \frac{\text{e}^{\lambda r}}{(-r)^{1-\beta}} \|\omega(r)\|_{X_{\beta}}\txtd r~\txtd s\\
	\leq &\sigma \widetilde{C}_{1-\beta}\int\limits_{0}^{t} \text{e}^{-(\lambda -cC_{F})(t-s)} ~\txtd s\int\limits_{-\infty}^{0} \frac{\text{e}^{\lambda r}}{(-r)^{1-\beta}} \|\omega(r)\|_{X_{\beta}}\txtd r=\frac{\sigma \widetilde{C}_{1-\beta}}{\lambda-cC_F}\int\limits_{-\infty}^{0} \frac{\text{e}^{\lambda r}}{(-r)^{1-\beta}} \|\omega(r)\|_{X_{\beta}}\txtd r.
	\end{align*}	
	In conclusion, using all the previous estimates in~\eqref{estimate:absorbing} we have
	\begin{align}\label{est_abs_time}
	\begin{split}
	\|u(t)\|_X  \leq& e ^{-\lambda t} \Bigg(c \|u_{0}\|_X + \sigma \hat c \|\omega(-t)\|_{X_{\beta}} \Bigg) + \frac{c\overline{C}_{F}}{\lambda}
	+\sigma \widetilde{C}_{1-\beta}\int\limits_{-\infty}^{0} \frac{\text{e}^{\lambda r}}{(-r)^{1-\beta}} \|\omega(r)\|_{X_{\beta}}\txtd r\\
	& + c\text{e}^{-(\lambda-cC_{F})t} \|u_{0}\|_X + \frac{c^2C_{F} \overline{C}_{F}} {\lambda(\lambda -cC_{F})} + cC_{F}\sigma\hat c \text{e}^{-(\lambda-cC_{F})t}\int\limits_{-t}^{0} \text{e}^{cC_{F}s}\|\omega(s)\|_{X_{\beta}}~\txtd s\\
	& + \frac{cC_{F}{\sigma}\widetilde{C}_{1-\beta}}{\lambda-cC_{F}} \int\limits_{-\infty}^{0}\frac{\text{e}^{\lambda s}}{(-s)^{1-\beta}}\|\omega(s)\|_{X_{\beta}}~\txtd s\\
	\leq &e ^{-(\lambda-cC_{F})t} \Bigg(2c \|u_{0}\|_X + \sigma \hat c \|\omega(-t)\|_{X_{\beta}} \Bigg) +\frac{c \overline{C}_F}{\lambda - cC_{F}}\\
	&   + cC_{F}\sigma\hat c \int\limits_{-\infty}^{0} \text{e}^{cC_{F}s}\|\omega(s)\|_{X_{\beta}}~\txtd s
 + \frac{\sigma\widetilde{C}_{1-\beta}\lambda}{\lambda-cC_{F}} \int\limits_{-\infty}^{0}\frac{\text{e}^{\lambda s}}{(-s)^{1-\beta}}\|\omega(s)\|_{X_{\beta}}~\txtd s.
\end{split}
	\end{align}
	Using~\eqref{criterion} we infer that 
	$B(\omega):=B(0,\rho(\omega)+\delta)$ for some $\delta>0$, where
	\begin{align*}
	\rho(\omega):=  \frac{c \overline{C}_F}{\lambda - cC_{F}} 
	+ cC_{F}\sigma\hat c \int\limits_{-\infty}^{0} \text{e}^{cC_{F}s}\|\omega(s)\|_{X_{\beta}}~\txtd s+
	 \frac{\sigma\widetilde{C}_{1-\beta}\lambda}{\lambda-cC_{F}} 
	\int\limits_{-\infty}^{0} \frac{\text{e}^{\lambda s}}{(-s)^{1-\beta}}\|\omega(s)\|_{X_{\beta}}~\txtd s
	\end{align*}
	is a pullback absorbing set for our random dynamical system.
	 This expression is natural, since we can 
	immediately see the influence of the linear part, nonlinear term and noise intensity. The previous integrals are well-defined due to the sub-exponential growth of $\omega$. More precisely, the set of all $\omega\in\Omega$ that have sub-exponential growth is invariant and has full measure.
	The temperedness of the absorbing set can be verified as in~\cite[Lem. 3.7]{BessaihGarridoSchmalfuss}.
	\qed
	\end{proof}

\subsection{Existence and finite fractal dimension of random attractors}

We now apply Theorem \ref{thm_rexpa} to deduce the existence of nonautonomous random exponential attractors for the random dynamical system $\varphi$.

\begin{theorem}\label{thm_expattr}
	For every $\nu\in(0,\frac{1}{2})$ and $\eta\in(0,1)$ the random dynamical system $\varphi$ generated by~\eqref{equation} has a nonautonomous
	random pullback exponential attractor $\mathcal{M}^{\nu,\eta}$, and its fractal dimension is bounded by 
	$$
	\sup_{t\in\R}\textnormal{dim}_f(\mathcal{M}^{\nu,\eta}(t,\omega))\leq \log_{\frac{1}{2\nu}}\left(N_{\frac{\nu}{\kappa}}^X(B^{X_\eta}(0,1))\right),
	$$
	where  
	$$
	\kappa=\widetilde{C}_\eta +C_F\widetilde{C}_\eta c\text{e}^{\frac{cC_F}{\lambda}}\int_0^{\tilde t}
	\frac{\text{e}^{-\lambda(\tilde t-s)}}{(\tilde t-s)^\eta}~\txtd s,
	$$
	and $\tilde t>0$ is arbitrary.
\end{theorem}

We remark that  $\kappa$ is determined by the constant
 $\widetilde{C}_\eta$ in Lemma \ref{lem_estimates}, 
the Lipschitz constant $C_F$ of $F$ and the constants $c$ and $\lambda$ in 
\eqref{omega:constant}.

	\begin{proof}
	We verify the hypotheses in Theorem \ref{thm_rexpa}.
	\begin{itemize}
	\item[$(H_1)$]  
	By {\bf Assumptions 2} (AC), this property holds for the spaces $X$ and $Y=X_\eta$, for arbitrary $\eta\in(0,1).$ 
	\item[$(H_2)$] 
	This was shown in Lemma \ref{lem_abs}. In fact, $B(\omega)=B(0,\rho(\omega)+\delta)$, for some $\delta>0$, 
	is pullback $\mathcal{D}$-absorbing and $B\in\mathcal{D}.$ 
	Moreover, the absorbing time fulfills the condition in $(H_1)$. 
	In fact, let $D\in \mathcal{D}$, then 
	$$
	T_{D,\omega}=\inf\left\{\tilde t\geq 0 : \text{e}^{-(\lambda -cC_F)t}\Big(2c\sup_{\zeta\in D(\theta_{-t}\omega)} \|\zeta \|+\sigma\hat c\|\omega(-t)\|_{X^\beta}\Big)<\delta\quad  \forall t\geq \tilde t\right\},
	$$
	see the estimate in \eqref{est_abs_time}.
	\item[$(H_4)$] We verify the Lipschitz continuity of $\varphi$ in $B$. To this end let $u_0,v_0\in B(\omega)$, $\omega\in\Omega$. For the difference of the corresponding solutions we obtain
\begin{align*}
	&\|\varphi( t, \omega, u_0)-\varphi(t, \omega, v_0)\|_{X}\\
	\leq &\|\widetilde{U}( t,\omega)(u_0-v_0)\|_{X}
	+\int_0^{t}\|\widetilde{U}( t-s,\theta_{s}\omega)
	\big(F(\varphi(s, \omega, u_0))-F(\varphi(s, \omega, v_0))\big)\|_{X}~\txtd s\\
	\leq &c \text{e}^{-\lambda t}\|u_0-v_0\|_X+\int_0^{t}c\text{e}^{-\lambda (t-s)}
	\big\|F(\varphi(s, \omega, u_0))-F(\varphi(s, \omega, v_0))\big\|_{X}~\txtd s\\
	\leq &c \text{e}^{-\lambda t}\|u_0-v_0\|_X+cC_F\int_0^{t}
	 \text{e}^{-\lambda (t-s)}
	\big\|\varphi(s, \omega, u_0)-\varphi(s, \omega, v_0)\big\|_{X}~\txtd s\\
	\leq &c \|u_0-v_0\|_X+cC_F\int_0^{t}
	 \text{e}^{-\lambda (t-s)}
	\big\|\varphi(s, \omega, u_0)-\varphi(s, \omega, v_0)\big\|_{X}~\txtd s.
	\end{align*}
	Hence, Gronwall's lemma implies that 
	$$
	\|\varphi( t, \omega, u_0)-\varphi(t, \omega, v_0)\|_{X}\leq 
	c \|u_0-v_0\|_X\text{e}^{cC_F\int_0^t\text{e}^{-\lambda(t-s)}~\txtd s}\leq 
	c\text{e}^{\frac{cC_F}{\lambda}} \|u_0-v_0\|_X.
	$$	
	\item[$(H_3)$] Finally, we use the Lipschitz continuity in $(H_4)$ to verify the smoothing property for the spaces $X$ and $Y=X_\eta$.
	Let $\tilde t>0$. We estimate the 
	difference of the solutions in the $X_\eta$-norm, 
	\begin{align*}
	&\|\varphi(\tilde t, \omega, u_0)-\varphi(\tilde t, \omega, v_0)\|_{X_\eta}\\
	\leq &\|\widetilde{U}(\tilde t,\omega)(u_0-v_0)\|_{X_\eta}
	+\int_0^{\tilde t}\|\widetilde{U}(\tilde t-s,\theta_{s}\omega)
	\big(F(\varphi(s, \omega, u_0))-F(\varphi(s, \omega, v_0))\big)\|_{X_\eta}~\txtd s\\
	\leq &\frac{\widetilde{C}_\eta}{\tilde{t}^\eta} \text{e}^{-\lambda\tilde t}\|u_0-v_0\|_X+\widetilde{C}_\eta\int_0^{\tilde t}\frac{\text{e}^{-\lambda(\tilde t-s)}}{(\tilde t-s)^\eta}
	\big\|F(\varphi(s, \omega, u_0))-F(\varphi(s, \omega, v_0))\big\|_{X}~\txtd s\\
	\leq &\frac{\widetilde{C}_\eta}{\tilde{t}^\eta} \text{e}^{-\lambda\tilde t}\|u_0-v_0\|_X+C_F\widetilde{C}_\eta\int_0^{\tilde t}
	\frac{\text{e}^{-\lambda(\tilde t-s)}}{(\tilde t-s)^\eta}
	\big\|\varphi(s, \omega, u_0)-\varphi(s, \omega, v_0)\big\|_{X}~\txtd s\\
	\leq &\frac{\widetilde{C}_\eta}{\tilde{t}^\eta}\|u_0-v_0\|_X+C_F\widetilde{C}_\eta c\text{e}^{\frac{cC_F}{\lambda}}\int_0^{\tilde t}
	\frac{\text{e}^{-\lambda(\tilde t-s)}}{(\tilde t-s)^\eta}~\txtd s
	 \|u_0-v_0\|_X,
	\end{align*}
	where we used the Lipschitz continuity of $\varphi$ in $X$ in the last step.
	Hence, the smoothing property holds with the smoothing constant
	$$
	\kappa=\frac{\widetilde{C}_\eta}{\tilde{t}^\eta} +C_F\widetilde{C}_\eta c\text{e}^{\frac{cC_F}{\lambda}}\int_0^{\tilde t}
	\frac{\text{e}^{-\lambda s}}{s^\eta}~\txtd s.
	$$
	The smoothing property holds for any $\tilde t>0$ and consequently, $(H_3)$ is satisfied.
	\end{itemize}
	\end{proof}
	
	An immediate consequence is the existence and 
finite fractal dimension of the global random attractor.

\begin{corollary}
There exists a unique global random attractor for $\varphi$ and its fractal dimension is
bounded by 
$$
	\textnormal{dim}_f(\mathcal{A}(\omega))\leq \inf_{\nu\in(0,\frac{1}{2})}\left\{\log_{\frac{1}{2\nu}}\left(N_{\frac{\nu}{\kappa}}^X(B^{X_\eta}(0,1))\right)\right\},
	$$
	for all $\eta\in(0,1)$, where $\kappa$ is the constant given in Theorem \ref{thm_expattr}.
\end{corollary}

The existence of a nonautonomous random exponential attractor implies that the global random attractor exists and that its fractal dimension is finite. 
We point out that in our particular case  
it is, in fact, easier to consider random exponential attractors than to 
deduce the existence of the global random attractor from Theorem \ref{attractorbounded}.
In fact, we have shown that a tempered absorbing set exists, but the theorem 
requires the existence of a compact absorbing set.
To show how this can be established and to indicate that the proof is indeed more involved 
than verifying the hypotheses of Theorem \ref{thm_rexpa} we provide the following lemma, even 
if we do not use it to prove our main results.

\begin{lemma}\label{compact}
Let $T_B\geq 0$ denote the absorbing time corresponding to the absorbing set $B$. Then, 
	the set $K(\omega):=\overline{\varphi(T_B,\theta_{-T_{B}}\omega, B(\theta_{-T_B}\omega)) }^{X}$ is a compact absorbing set for $\varphi$.
\end{lemma}

\begin{proof}
The proof is based on compact embeddings of fractional power spaces. Let $\eta>0$ be such that $0<\eta<\beta\leq 1$. It suffices to derive uniform estimates of the solutions w.r.t. the $X_{\eta}$-norm since $X_{\eta}$ is compactly embedded into $X$. Let $u_0\in B(\theta_{-T_B}\omega)$. We observe that
	\begin{align*}
	\|\varphi(T_B,\theta_{-T_B}\omega,u_{0})\|_{X_{\eta}}  \leq &\| \widetilde{U}(T_B,\theta_{-T_B}\omega)u_{0}\|_{X_{\eta}} + \|\widetilde{U}(T_B,\theta_{-T_B}\omega)\theta_{-T_B}\omega(T_B)\|_{X_{\eta}} \\
&	+ \int\limits_{0}^{T_B} \|\widetilde{U}(T_B-s,\theta_{s-T_B}\omega) F(\varphi(s,\theta_{-T_B}\omega,u_{0})) \|_{X_{\eta}}~\txtd s \\
	& +\sigma\int\limits_{0}^{T_B} \|\widetilde{U}(T_B-s,\theta_{s-T_B}\omega) A(\theta_{s-T_B}\omega) \omega(s-T_B) \|_{X_{\eta}}~\txtd s.
	\end{align*}
	To estimate these terms we use that $u_{0}\in B(\theta_{-T_B}\omega)$ and that $B$ is a pullback absorbing set.
	The first and second term yield the following expressions,
	\begin{align*}
	 \|\widetilde{U}(T_B,\theta_{-T_B}\omega)u_{0}\|_{X_{\eta}} \leq \widetilde{C}_\eta \text{e}^{-\lambda T_B} \|u_{0}\|_X\leq \widetilde{C}_\eta e ^{-\lambda T_B} (\rho(\theta_{-T_B}\omega)+\delta),
	\end{align*}
	and
	\begin{align*}
	  \|\widetilde{U}(T_B,\theta_{-T_B}\omega)\omega(-T_B)\|_{X_{\eta}}\leq \widetilde{C}_\eta \frac{\hat c}{c}\text{e}^{-\lambda T_B}\|\omega(-T_B)\|_{X_{\beta}}.
	\end{align*}
	For the generalized convolution we obtain
	\begin{align*}
	&\int\limits_{0}^{T_B} \|\widetilde{U}(T_B-s,\theta_{s-T_B}\omega) A(\theta_{s-T_B}\omega) \omega(s-T_B) \|_{X_{\eta}}~\txtd s \\
	=&\int\limits_{0}^{T_B} \|A^{\eta}(\theta_{T_B}\omega)\widetilde{U}(T_B-s,\theta_{s-T_B}\omega)A^{1-\beta}(\theta_{s-T_B}\omega)A^{\beta}(\theta_{s-T_B}\omega)\omega(s-T_B) \|_{X}~~\txtd s\\
	\leq &\widetilde{C}_{1-(\beta-\eta)}\int \limits_{0}^{T_B} \frac{\text{e}^{-\lambda(T_B-s)}}{(T_B-s)^{1-(\beta-\eta)}} \|\omega(s-T_B)\|_{X_{\beta}} ~\txtd s\\
		\leq&\widetilde{C}_{1-(\beta-\eta)} \sup\limits_{s\in[0,T_B]} \|\omega(s-T_B)\|_{X_{\beta}}\int\limits_{0}^{T_B} \frac{\text{e}^{-\lambda(T_B-s)}}{(T_B-s)^{1-(\beta-\eta)}}~\txtd s<\infty .
	\end{align*}
	Finally, we estimate the drift term,
	\begin{align*}
	&\int\limits_{0}^{T_B} \|\widetilde{U}(T_B-s,\theta_{T_B-s}\omega) F(\varphi(s,\theta_{-T_B}\omega,u_{0})) \|_{X_{\eta}}~\txtd s 
	\\ \leq 
	&\widetilde{C}_\eta\overline{C}_{F}\int\limits_{-T_B}^{0}\frac{\text{e}^{\lambda r}}{(-r)^{\eta}}~\txtd r 
	+ \widetilde{C}_\eta C_{F} \int\limits_{-T_B}^{0} \frac{\text{e}^{\lambda r}}{(-r)^{\eta}} \|\varphi(r+T_B,\theta_{-T_B}\omega,u_{0})\|_X~\txtd r\\
	\leq 
	&\widetilde{C}_\eta\overline{C}_{F}\int\limits_{-T_B}^{0} \frac{\text{e}^{\lambda r}}{(-r)^{\eta}}~\txtd r
	+ \widetilde{C}_\eta C_{F} \int\limits_{-T_B}^{0} \frac{\text{e}^{\lambda r}}{(-r)^{\eta}}(\rho(\theta_r\omega)+\delta)~\txtd r,
	\end{align*}
	where we used that $u_0\in B(\theta_{-T_B}\omega)$ and the absorbing property of $B$, i.e., 
	$$
	\varphi(r+T_B,\theta_{-T_B}\omega,u_{0})\subset B(\theta_r\omega).
	$$ 
	We remark that the expressions and $\omega$-dependent constants in all estimates are well-defined. 
 	Collecting the estimates, we finally conclude that
	\begin{align*}
	\|\varphi(T_B,\theta_{-T_B}\omega, u_{0})\|_{X_{\eta}} \leq \widetilde{C}(\omega, \eta,\beta,\delta, T_B)<\infty,
	\end{align*}
	for some constant $\widetilde{C}(\omega, \eta,\beta,\delta, T_B).$
	The compact embedding $X_{\eta}\hookrightarrow X$ implies that the set 
	$\varphi(T_B,\theta_{-T_B}\omega, B(\theta_{-T_{B}}\omega))$ is relatively compact in $X$. Hence, $K(\omega)$ is compact which proves the statement.
	\qed
	\end{proof}\\

\section{Examples}\label{examples}

We provide examples of differential operators $A$ that satisfy our assumptions 
in the previous sections. The canonical examples are  uniformly elliptic operators with random, time-dependent coefficients. Such operators have been investigated in the context of SPDEs, see \cite{PronkVeraar, PortalVeraar} and the references therein. In the framework of random dynamical systems several properties of the evolution system generated by such operators have been analysed, including results on the spectral theory and principal Lyapunov exponents~\cite{book,MShen, ShenVickers}, stable and unstable manifolds and multiplicative ergodic theorems~\cite{CDLS,lns}.\\
 
We consider the random partial differential operators
in the Banach space $X:=L^{p}(G)$ for $2\leq p<\infty$, where $G\subset\mathbb{R}^{n}$ is a bounded open domain with smooth boundary $\partial G$ (see also~\cite[Sec.~7.6]{pazy}).  We recall that in our case the differential operator $A(t,\omega)$
depends on time $t\in\mathbb{R}$ and  the random parameter $\omega\in\Omega$ in the 
following way
$$ A(t,\omega)=A(\theta_t\omega), \qquad t\in\R,\omega\in\Omega.$$

\begin{example}\label{ex1}
	Let $m\in\N$ and $A$ be the random partial differential operator 
	$$
	A(\theta_t\omega,x,\txtD)=\sum\limits_{|k|\leq 2m}a_{k}(\theta_t\omega,x)\txtD^{k},\qquad t\in\R,~ \omega\in\Omega,~ x\in G,
	$$
	with homogeneous Dirichlet boundary conditions, 
	$$
	\txtD^{k}u=0\quad\text{on }\partial G \qquad\text{ for } |k|<m.
	$$ 
	We make the following assumptions on the coefficients. 
	\begin{enumerate}
		\item The operator $A$ is uniformly strongly elliptic in $G$, i.e.~there exists a constant $\overline{c}>0$ such that 
		$$ 
		(-1)^m	\sum\limits_{|k|=2m}a_{k}(\theta_t\omega,x)\xi_{k}\geq \overline{c}|\xi|^{2m}\quad \mbox{ for all } t\in\mathbb{R},~ x\in\overline{G},~ \xi\in\mathbb{R}^{n}.$$
		\item The coefficients form a stochastic process $(t,\omega)\mapsto a_{k}(\theta_{t}\omega,\cdot{})\in C^{2m}(\overline{G})$ which has H\"older continuous trajectories, i.e. there exists $\nu\in(0,1]$ such that 
		\begin{align}\label{hcoeff}
		|a_{k}(\theta_{t}\omega,x)-a_{k}(\theta_{s}\omega,x)|\leq \overline{c}_{1}|t-s|^{\nu}\quad\mbox{ for all } t, s\in\mathbb{R},~ x\in\overline{G},~  |k|\leq 2m,
		\end{align}
		for some $\overline{c}_{1}>0$.
		\item The constants $\overline{c}$ and $\overline{c}_1$ are uniformly bounded with respect to $\omega\in\Omega$ and $x\in G$. 
	\end{enumerate}
	We define the $L^{p}$-realization of $A(\cdot{},\cdot, D)$ by
	\begin{align}\label{defop}
	&A_{p}(\theta_t\omega)u:=A(\theta_t\omega,x,\txtD)u\qquad  \mbox{ for }
	u\in \mathcal{D}_A, \mbox{ where }\\
	&\mathcal{D}_A:=D(A_{p}(\theta_t\omega))=W^{2m,p}(G)\cap W^{m,p}_{0}(G).\nonumber 
	\end{align} 
	
	We verify now Assumptions (A0)--(A3) and (AC). It is well-known that $A_{p}(\omega)$ generates a compact analytic semigroup in $L^p(G)$ for every $\omega\in\Omega$. Moreover, the mapping $\omega\mapsto A_p(\omega)v$ is measurable for every smooth function $v\in C^{\infty}(G)$. This entails the measurability of the mapping $\omega\mapsto A_p(\omega)v$ for every $v\in L^{p}(G)$. Consequently, the assumptions (A0)-(A2) and (AC) are satisfied.
	We only need to show the H\"older continuity of the mapping $t\mapsto A_{p}(\theta_{t}\omega)$ to verify (A3).\\
	
	To this end,  let $v\in \cD_A$ and $t,s\in\R$. Then, we have
	\begin{align*}
	\|A_{p}(\theta_{t}\omega)-A_{p}(\theta_{s}\omega)\|^{p}_{\mathcal{L}(\cD_{A},X)}& =\sup\limits_{v\in\mathcal{D}_A,\|v\|=1}\|(A_{p}(\theta_{t}\omega)-A_{p}(\theta_{s}\omega))v\|_{X}^{p}\\
	& = \sup\limits_{v\in\mathcal{D}_A,\|v\|=1} \Big\|\sum\limits_{|k|\leq 2m}(a_{k}(\theta_{t}\omega,x)-a_{k}(\theta_{s}\omega,x))\txtD^{k}v\Big\|_{X}^{p}.
	\end{align*}	
	Furthermore, we estimate
	\begin{align*}
	&\Big\|\sum\limits_{|k|\leq 2m}(a_{k}(\theta_{t}\omega,x)-a_{k}(\theta_{s}\omega,x))\txtD^{k}v\Big\|_{X}^{p}\\
	=&\  \int\limits_{G}\Bigg|\sum\limits_{|k|\leq 2m} (a_{k}(\theta_{t}\omega,x) -a_{k}(\theta_{s}\omega,x))D^{k}v(x)\Bigg|^{p}~\txtd x\\
	\leq &\  \int\limits_{G} C_{p} \sum\limits_{|k|\leq 2m} |(a_{k}(\theta_{t}\omega,x)-a_{k}(\theta_{s}\omega,x))D^{k} v(x) |^{p}~\txtd x\\
	\leq &\ C_{p}\sum\limits_{|k|\leq 2m}\sup\limits_{x\in G} |a_{k}(\theta_{t}\omega,x)-a_{k}(\theta_{s}\omega,x)|^{p} \int\limits_{G}|D^{k}v(x)|^{p}~\txtd x
	\\
	\leq &\ C_{p}\sum\limits_{|k|\leq 2m}\sup\limits_{x\in G} |a_{k}(\theta_{t}\omega,x)-a_{k}(\theta_{s}\omega,x)|^{p} \|v\|^{p}_{W^{2m,p}}\\
	\leq &\ C_{p} ||v||^{p}_{W^{2m,p}}\sum\limits_{|k|\leq 2m}\|a_{k}(\theta_{t}\omega, \cdot)-a_{k}(\theta_{s}\omega, \cdot)\|^{p}_{C^{2m}(\overline{G})}.
	\end{align*}
	Therefore, the H\"older continuity of $(t,\omega)\mapsto a_{k}(\theta_{t}\omega,\cdot{})$
	justifies~\eqref{kt}.
	
\end{example}

The following example is similar to~\cite[Example~6.2]{PronkVeraar} and~\cite[Section 10.2]{Yagi}. In our case, the operators satisfy the structural assumption \eqref{struc_ass} and
the  domains are assumed to be constant with respect to time $t$ and $\omega\in\Omega$. 
For random nonautonomous second order operators of this type see~\cite[Sec.~3]{MShen} and~\cite{book}.

\begin{example} 
	Let $m\in\N$ and $A$ be the differential operator 
	$$ 
	A(\theta_t\omega,x,\txtD):=\sum\limits_{|k_1|,|k_2|\leq m}\txtD^{k_1}(a_{k_1,k_2}(\theta_t\omega,x)\txtD^{k_2}),\qquad \omega\in\Omega,~ x\in G,~t\in\mathbb{R},
	$$ 
	with homogeneous Dirichlet or Neumann boundary conditions. Similarly as in the previous example and~\cite{PronkVeraar} we make the following assumptions on the coefficients.
	\begin{enumerate}
		\item We assume that the coefficients $a_{k_1,k_2}$ are bounded and symmetric. More precisely, there exists a constant $K\geq 1$ such that 
		$$
		|a_{k_1,k_2}(\theta_t\omega,x)|\leq K\qquad \text{for all}\ |k_1|,|k_2|\leq m,\ t\in\mathbb{R}, x\in G, \omega\in\Omega
		$$ 
		and $$a_{k_1,k_2}(\cdot,\cdot)=a_{k_2,k_1}(\cdot,\cdot)~~\mbox{ for }~|k_1|,|k_2|\leq m.$$ Furthermore, the mapping $t\mapsto \txtD^{k} a_{k_1,k_2}(\theta_t\omega,x)$ is continuous for $|k|,|k_1|,|k_2|\leq m$, $\omega\in\Omega$ and $x\in G$.
			\item The operator $A$ is uniformly elliptic in $G$, i.e. there exists a constant $\bar c>0$ such that 
		$$ 
		\sum\limits_{|k_1|=|k_2|=m}a_{k_1,k_2}(\theta_t\omega,x)\xi_{k_1}\xi_{k_2}\geq \overline{c}|\xi|^{2m}\qquad \mbox{ for all } t\in\mathbb{R},~x\in\overline{G},~ \xi\in\mathbb{R}^{n}.
		$$
		\item The coefficients form a stochastic process $(t,\omega)\mapsto a_{k_1,k_2}(\theta_{t}\omega,\cdot{})\in C^{m}(\overline{G})$ with H\"older continuous trajectories as in~\eqref{hcoeff}. This means that there exists  $\nu\in(0,1]$ such that 
					\begin{align*}
		|a_{k_1,k_2}(\theta_{t}\omega,x)-a_{k_1,k_2}(\theta_{s}\omega,x)|\leq \overline{c}_{2}|t-s|^{\nu}\quad\mbox{ for all } t\in\mathbb{R},~ x\in\overline{G},~  |k_1|, |k_2|\leq m,
		\end{align*}
		for some constant $\overline{c}_{2}>0$.
			\item All constants $K$, $\overline{c}$ and $\overline{c}_2$ are uniformly bounded with respect to $\omega\in\Omega$ and $x\in G$. 
	\end{enumerate} 
	 One can define the $L^{p}$-realization $A_p$ of $A(\cdot,\cdot,\txtD)$ as in \eqref{defop}. Moreover, one can verify as in Example~\ref{ex1} that the assumptions (A0)--(A3) are satisfied. 
	For instance, (A3) follows from the estimate
	\begin{align*}
	&\|A_{p}(\theta_{t}\omega)-A_{p}(\theta_{s}\omega)\|^{p}_{\mathcal{L}(\cD_{A},X)} =\sup\limits_{v\in\mathcal{D}_A,\|v\|=1}\|(A_{p}(\theta_{t}\omega)-A_{p}(\theta_{s}\omega))v\|_{X}^{p}\\
	&=\Bigg\|\sum\limits_{|k_1|,|k_{2}|\leq m} \txtD^{k_1}(a_{k_1,k_2}(\theta_{t}\omega,x)-a_{k_1,k_2}(\theta_{s}\omega,x))\txtD^{k_2}v\Bigg\|^{p}_{X}\\
	&=\int\limits_{G}\left|\sum\limits_{|k_1|,|k_2|\leq m} \txtD^{k_1} (a_{k_1,k_2}(\theta_{t}\omega,x)-a_{k_1,k_2}(\theta_{s}\omega,x))\txtD^{k_2}v(x) \right|^{p}~\txtd x\\
	& \leq C_{p}\sum\limits_{|k_1|,|k_2|\leq m} \int\limits_{G} |\txtD^{k_1}(a_{k_1,k_2}(\theta_{t}\omega,x)-a_{k_1,k_2}(\theta_{s}\omega,x))\txtD^{k_2}v(x)|^{p}~\txtd x\\
	&\leq C_{p} \sum\limits_{|k_1|,|k_2|\leq m} \sup\limits_{x\in G} |\txtD^{k_1}(a_{k_1,k_2}(\theta_{t}\omega,x)-a_{k_1,k_2}(\theta_{s}\omega,x))|^{p}\int\limits_{G}|\txtD^{k_2}v(x)|^{p}~\txtd x \\
	&\leq C_{p} \|v\|^{p}_{W^{2m,p}}\sum\limits_{|k_1|,|k_2|\leq m} \|a_{k_1,k_2}(\theta_{t}\omega,\cdot)-a_{k_1,k_2}(\theta_{s}\omega,\cdot)\|^{p}_{C^{m}(\overline{G})}.
	\end{align*}
\end{example}

\begin{example}
	Another widely studied  example are operators of the form $A:=\Delta + a(\theta_t\omega)$, where $\Delta$ denotes the Laplace operator with homogeneous Dirichlet boundary conditions 
	in a bounded domain $G\subset\R^n$ and $a(\theta_t\omega)$ can be viewed as a time-dependent random potential~\cite{SalakoShen}. Here, the function $a:\Omega\to (0,\infty)$ is measurable and the mapping $(t,\omega)\mapsto a(\theta_t\omega)$ is assumed to be H\"older continuous. For instance, in mathematical models for populations dynamics such random potentials can be used to quantify environmental fluctuations, see e.g.~\cite{tomasso} and the references specified therein. Several PDEs where the linear part has this structure have been investigated, see e.g.~\cite{SalakoShen} where a random nonautonomous version of the Fisher-KPP equation is considered. The asymptotic dynamics of the solutions as $t$ tends to infinity
has been characterized depending on the behavior of $a$. 
\end{example}

\subsection*{Acknowledgements}
We would like to thank the anonymous referees for their valuable comments and remarks. 
CK and AN have been supported by a German Science Foundation (DFG) grant in the D-A-CH framework, grant KU 3333/2-1~(AN until 31. December 2019). CK acknowledges partial support by a Lichtenberg Professorship funded by the VolkswagenStiftung and by the EU within the TiPES project funded the European Unions Horizon 2020 research and innovation programme under Grant Agreement No. 820970.


\end{document}